\documentclass[12pt]{amsart}
\usepackage{amsmath, amsfonts, amssymb}
\newcommand{\dal}{\square}
\newcommand{\R}{{\mathbb R}}
\newcommand{\N}{{\mathbb N}}
\newcommand{\ve}{\varepsilon}
\newcommand{\pa}{\partial}

\newcommand{\jb}[1]{\langle #1 \rangle}
\newcommand{\jbx}{\jb{x}}
\newcommand{\norm}[2]{\|#1 \!:\! #2\|}
\newcommand{\Do}{\Omega}

\newcommand{\TN}[2]{|\!|\!| #1 |\!|\!|_{#2}} 

\DeclareMathOperator{\supp}{supp}
\newtheorem{theorem}{Theorem}
\newtheorem{lm}{Lemma}
\newtheorem{prop}[lm]{Proposition}
\newtheorem{cor}[lm]{Corollary}
\newtheorem{remark}[lm]{Remark}
\numberwithin{equation}{section}
\numberwithin{theorem}{section}
\numberwithin{lm}{section}
\begin{document}
\title[Lifespan for nonlinear wave equations]
{Lower bound of the lifespan of solutions to
 semilinear wave equations in an exterior domain}

\author[S.~Katayama]{Soichiro Katayama}
\address{Department of Mathematics, Wakayama University, 930 Sakaedani, Wakayama 640-8510, Japan}
\email{katayama@center.wakayama-u.ac.jp}
\author[H.~Kubo]{Hideo Kubo}
\address{Division of Mathematics,
Graduate School of Information Sciences,
Tohoku University,
Sendai 980-8579, Japan}
\email{kubo@math.is.tohoku.ac.jp}
\dedicatory{Dedicated to Professor Akitaka Matsumura on the occasion of his
$60$th birthday}
\date{}
\begin{abstract}
We consider the Cauchy-Dirichlet
problem for semilinear wave equations
in a three space dimensional domain
exterior to a bounded and non-trapping obstacle.
We obtain a detailed estimate for 
the lower bound of the lifespan of classical solutions
when the size of initial data tends to zero, in a similar spirit
to that of the works of John and H\"ormander where the Cauchy problem 
was treated.
We show that our estimate is sharp at least for some special case.
\end{abstract}
\maketitle
\section{Introduction}
\label{sec:1}
In this paper we consider 
the mixed problem for semilinear wave equations in an exterior domain 
$\Omega(\subset \R^3)$ with compact and smooth boundary $\pa \Omega$:
\begin{align}\label{eq}
& \dal u =F(\pa u), & (t,x) \in (0,T)\times \Omega,
\\ \label{dc}
& u(t,x)=0, & (t,x) \in (0,T)\times \partial\Omega,
\\ \label{id}
& u(0,x)=\ve f_0(x),\ (\partial_t u)(0,x)=\ve f_1(x), & x\in \Omega,
\end{align}
where $\dal=\pa_t^2-\Delta_x$ with the Laplacian $\Delta_x=\sum_{j=1}^3\pa_{x_j}^2$,
$\pa u=(\pa_t u, \nabla_x u)$, and
$\ve$ is a small and positive parameter.
We suppose that $\Omega=\R^3 \setminus \overline{\mathcal O}$ with
a bounded open set ${\mathcal O}(\subset \R^3)$.
Throughout this paper, we assume that ${\mathcal O}$ is a non-trapping obstacle
(see, e.g., Melrose~\cite{Mel79} and Lax-Phillips~\cite[Epilogue]{LaPh}
for the definition;
for example a star-shaped obstacle is known to be non-trapping).
Without loss of generality, we may also assume
\begin{equation}
{\mathcal O}\subset \{x\in \R^3\,;\, |x|\le 1\}.
\end{equation}
Suppose that 
\begin{align}\label{0.3}
F(\pa u)=\sum_{a,b=0}^3 g^{a,b} 
 (\partial_a u) (\partial_b u) 
\end{align}
with real constants $g^{a,b}$,
where $\pa_0=\pa/\pa t$, and $\pa_j=\pa/\pa x_j$ for $j=1,2,3$.
To avoid a complicated discussion on 
the compatibility condition, 
we assume that $\vec{f}=(f_0,f_1)\in
\bigl(C^\infty_0(\Omega)\bigr)^2$ 
(we can relax this assumption; see Remark~\ref{GeneralData} below).

We are interested in the behavior of the lifespan $T_\ve$ 
of the solution $u$ to \eqref{eq}--\eqref{id} as $\ve \to +0$.
Here the lifespan $T_\ve$ is defined by
the supremum of all $T>0$ such that the mixed problem 
(\ref{eq})--(\ref{id})
admits a unique smooth solution $u$ in $[0, T) \times\Omega$.
When $t$ is small, it is natural to expect that $u(t,x)$ can be approximated by 
$\ve u_0(t,x)$, where $u_0$ is the solution to the
corresponding linear homogeneous wave equation\,:
\begin{align}\label{hw}
& \dal u_0(t,x)= 0, & (t,x) \in (0,T)\times \Omega,
 \\ \label{dch}
& u_0(t,x)=0, & (t,x) \in (0,T)\times \partial\Omega,
\\ \label{idh}
& (u_0(0,x), (\partial_t u_0)(0,x))=\vec{f}(x), & x\in \Omega.
\end{align}
This is in fact the case, as long as $\ve \log t \le A$
with a suitably small number $A$
(see, e.g., Keel, Smith and Sogge \cite{KeSmiSo02}, and Kubo \cite{Kub06}).
However, when $\ve \log t \ge A$, the solution $u(t,x)$
for the nonlinear problem might not stay close to the
free solution $\ve u_0(t,x)$.
For instance, if $F(\pa u)=(\pa_t u)^2$,
then one can find a class of initial data for which
the solution blows up in finite time (see Theorem~\ref{Upper} below).
Therefore, in order to understand the behavior of the solution 
$u(t,x)$ for $\ve \log t \ge A$,
we need to take the nonlinear effect into account so that a better approximate solution
can be constructed for such large values of $t$.
We construct such an approximate solution by solving a nonlinear ordinary differential equation 
\eqref{3.5} which is related to \eqref{eq} but is irrelevant to the presence of the obstacle.
The issue is to find a suitable 
initial data for the ordinary differential equation so that we can
match the nonlinear approximate solution
with the free solution $\ve u_0(t,x)$
that is relevant to the presence of the obstacle.
This will be done by making use of the scattering theory for the linear 
exterior problem:
As is well known, for each fixed $s\in \R$ and $\theta \in S^2$,
there exists a limit of $-t\pa_t u_0(t,x)$
along the ray $\{(t,x); |x|-t=s,\ x/|x|=\theta\}$ as $t\to \infty$ 
(see, e.g., \cite{LaPh}).
We write the limit as ${\mathcal F}_+(s, \theta)$; in other words
we define ${\mathcal F}_+(s,\theta)$ by
\begin{align}\label{DefRadiation}
{\mathcal F}_+(s,\theta):=\lim_{t\to \infty} (-t)(\pa_t u_0)
\bigl(t,(s+t)\theta\bigr)
\end{align}
for $s \in \R$ and $\theta \in S^2$.
In \cite{kkInverse} the rate of the convergence in \eqref{DefRadiation} was studied
(see \eqref{1.0} and \eqref{f-00} below).
Choosing $\mathcal{F}_+$ as the initial data
for the ordinary differential equation, 
we shall construct a good approximation
of the solution to the 
original mixed problem (see Section \ref{Approx} below for the details).
Once we have such an approximate solution, we are able to 
obtain a lower bound of the lifespan $T_\ve$ as $\ve \to +0$,
analogously to the case of the Cauchy problem
studied by H\"ormander \cite{Hoe87, Hoe97} and John \cite{Joh87}
independently, and also by 
others (see \cite{Ali01, God93, Hos00} for instance).

In order to state our result precisely, we set
\begin{align}\label{a}
G(\theta)=\sum_{a,b=0}^3 g^{a,b} \theta_a \theta_b 
\quad \mbox{for} \ \theta=(\theta_1,\theta_2, \theta_3)
  \in S^2,
\end{align}
where $g^{a,b}$ is the constant from (\ref{0.3}), and $\theta_0=-1$.
We also set
\begin{align}\label{DefTau}
\tau_*=\left(\sup \{-2^{-1}G(\theta) 
{\mathcal F}_+(s,\theta)\,;\,s \in \R, \ \theta \in S^2\,\} 
\right)^{-1}.
\end{align}
We can show that 
$\tau_*$ is a finite positive number 
when $G\not\equiv 0$ on $S^2$, and $\vec{f}\not\equiv 0$ on $\Omega$
(see Lemma~\ref{Triviality} below for the proof).
If $G\equiv 0$ on $S^2$, 
which is equivalent to the so-called {\it null condition}, 
then the global solvability is known for sufficiently small initial data
(see, e.g., Metcalfe, Nakamura and Sogge \cite{MetNaSo05b}, and the authors \cite{KatKub08MSJ});
we also have the global solution $u\equiv0$ if $\vec{f}\equiv 0$.

Our main result is the following:
\begin{theorem}\label{Th.1.1}
Assume that ${\mathcal O}$ is a non-trapping obstacle.
If $G\not\equiv 0$ on $S^2$,
then for any $\vec{f} \in (C^\infty_0(\Omega))^2$
with $\vec{f}\not\equiv 0$ we have
\begin{align}
\label{MainInf}
\liminf_{\ve \to +0} \ve \log T_\ve \ge \tau_*,
\end{align}
where $\tau_*$ is the number defined by \eqref{DefTau}.
\end{theorem}
As we have mentioned, analogous results were already 
obtained for the Cauchy problem (see \cite{Hoe87, Hoe97, Joh87}; more precisely
the quasi-linear case was treated in \cite{Hoe87} and \cite{Joh87},
while the semilinear case is also considered in \cite{Hoe97}).
When the initial data is radially symmetric, Godin \cite{God08}
obtained a similar result for quasi-linear equations
with the Neumann boundary condition.
In the case of the Cauchy problem, vector fields $t\pa_t+x\cdot\nabla_x$ and 
$t\pa_j+x_j\pa_t$ ($j=1,2,3$) were
effectively used to obtain the decay of the solution 
through Klainerman's inequality;
however these fields are not useful in the study of the mixed problem.
In our proof, to compensate the lack of those vector fields,
we use the weighted $L^\infty$-$L^\infty$ estimate
which involves only the standard derivatives and the generators of spatial rotations
(see Proposition~\ref{main} below);
careful treatments of the decay factor 
$(1+|t-|x|\,|)^{-1}$ are also needed.
Our method is also applicable to the case $\Omega=\R^3$.

Unfortunately we do not have the estimate in the opposite direction
to \eqref{MainInf}, that is to say
\begin{align}
\label{upperbound}
\limsup_{\ve \to +0} \ve \log T_\ve \le \tau_*
\end{align}
in general situations.
Note that the same is true for the semilinear Cauchy problem, 
though the estimate corresponding to
\eqref{upperbound} is obtained for the quasi-linear Cauchy problem
(see Alinhac~\cite{Ali01:02}). 
However the following result,  motivated by \cite{God93}, shows that we have \eqref{upperbound} for some special case.
\begin{theorem}\label{Upper}
Let ${\mathcal O}=\{x\in \R^3; |x|<1\}$, and $F=c(\pa_t u)^2$ with some constant $c\,(\ne 0)$.
We suppose that $f_0\equiv 0$, and $f_1\in C^\infty_0(\Omega)$.
If $f_1$ is non-zero and radially symmetric, and $cf_1$ is non-negative,
then we have \eqref{upperbound}. Consequently we have
$\lim_{\ve\to +0}\ve \log T_\ve=\tau_*$ \text{for such $\vec{f}=(f_0, f_1)$}.
\end{theorem}

This paper is organized as follows.
In the next section we gather notation.
In Section~\ref{Pre} 
we give some preliminaries.
The approximate solution will be constructed in Section~\ref{Approx}.
Section~\ref{proof} is devoted to the proof of Theorem~\ref{Th.1.1}
which is reduced to Lemma \ref{Lemma5.1}.
Theorem~\ref{Upper} will be proved in Section~\ref{BlowUp}.

\medskip

As usual, various positive constants
which may change line by line are denoted just by the same letter $C$ 
throughout this paper.

\section{Notation}
\label{Notation}
In this section, we introduce notation which will be used throughout this paper.

We write $\pa_0=\pa_t$ and $\pa_j=\pa_{x_j}$ for $j=1,2,3$.
We denote $ r=|x| $ and $\omega={x}/{r}$. 
We set $\pa_r=\sum_{j=1}^3 (x_j/r) \partial_j $ and
$ O_{ij}=x_i\pa_j-x_j\pa_i$ for $i, j=1, 2, 3$.
Then we have
\begin{align}
\label{rew1}
\nabla_x=\,\omega \pa_r - r^{-1} \omega \wedge O 
\quad\text{with}\ O=(O_{23},O_{31},O_{12}).
\end{align}
We denote $Z=\{Z_0, Z_1, \ldots, Z_{6} \}
=\{\pa_t, \pa_1, \pa_2, \pa_3, O_{23}, O_{31}, O_{12} \}$. 
We write $Z^\alpha$ for $Z_0^{\alpha_0}\cdots Z_{6}^{\alpha_{6}}$
with a multi-index $\alpha=(\alpha_0,\ldots, \alpha_{6})$.
Note that we have $[Z_a,\dal]=0$~($a=0,\dots,6$),
where we have set $[A,B]=AB-BA$.
For a non-negative integer $s$ and a smooth function $\varphi$,
we define
\begin{align}\nonumber
 |\varphi(t,x)|_s=\sum_{|\alpha|\le s} |(Z^\alpha \varphi)(t,x)|,
\quad
 |\pa \varphi(t,x)|_s=\sum_{|\alpha|\le s} \sum_{a=0}^3 |(Z^\alpha \pa_a \varphi)(t,x)|.
\end{align}

For functions of $(s,\theta,\tau) \in \R \times S^2 \times [0,\infty)$,
we denote the differentiation with respect to $s$, $\theta$ and $\tau$ by 
\begin{align}\label{3.8}
\Lambda_0=\partial_s, \quad \Lambda_1=o_{23}, 
\quad \Lambda_2=o_{31}, \quad \Lambda_3=o_{12},
\quad \Lambda_4=\partial_\tau,
\end{align}
where a differential operator $o_{ij}$ on $S^2$ is (formally) defined by $o_{ij}=\theta_i \partial_{\theta_j}-\theta_j \partial_{\theta_i}$.
We write $\Lambda^\beta$ for $\Lambda_0^{\beta_0}\cdots \Lambda_{4}^{\beta_4}$
with a multi-index $\beta=(\beta_0,\ldots, \beta_4)$.
For a multi-index $\gamma=(\gamma_0,\ldots,\gamma_3)$ we define
$\Lambda_*^\gamma=\Lambda_0^{\gamma_0}\cdots \Lambda_3^{\gamma_3}$.

We write $\jb{z}=\sqrt{1+|z|^2}$ for $z\in \R^d$, where $d$ is a positive integer.
For $\nu$, $\kappa \in \R$, we define 
\begin{align}
\label{DefPsi}
{\Psi}_\nu(t)= &
   \begin{cases}
        \log(2+t)  & \text{ if } \nu=0,
\\
       1
         & \text{ if } \nu\not=0,
   \end{cases}
\\
\label{defW}
W_{\nu,\kappa}(t,x)= &
\langle t+|x|\rangle^\nu \bigl( \min \left\{ \jb{x}, \jb{t-|x|} \right\} \bigr)^\kappa.
\end{align}
We also define
\begin{equation}\label{NfW}
 \|h(t)\!:\!{N_k({\mathcal W})}\|
=\sup_{(s,x) \in [0,t] \times \Do}
\jbx\,{\mathcal W}(s,x)\,|h(s,x)|_k
\end{equation}
for $t\in [0,T]$, a non-negative integer $k$, and a non-negative
function $\mathcal{W}(s,x)$.
In addition, for $\rho\ge 0$ and a non-negative integer $k$, we define 
\begin{equation}
{\mathcal A}_{\rho, k}[\vec{f}\,]=\sup_{x\in \Do} \, \jbx^{\rho} 
\bigl(|f_0(x)|_{k}+|\nabla_x f_0(x)|_k+|f_1(x)|_k\bigr)
\label{HomWei}
\end{equation}
for a smooth function $\vec{f}=(f_0, f_1)$ on $\Do$.

For $R>0$, $B_R(y)$ stands for the open ball in $\R^3$ with radius $R$
centered at $y\in \R^3$.

\section{Preliminaries}
\label{Pre}
First we derive some estimates for ${\mathcal F}_+$
introduced by \eqref{DefRadiation}.
${\mathcal S}(\R^3)$ denotes the set of rapidly decreasing functions.
For $\vec{\varphi}=(\varphi_0, \varphi_1)\in\bigl({\mathcal S}(\R^3)\bigr)^2$, we define
the Friedlander radiation field 
$$
{\mathcal F}_0[\vec{\varphi}](s,\theta)=\frac{1}{{4\pi}} 
  \{{\mathcal R}[\varphi_1](s,\theta)-\pa_s {\mathcal R}[\varphi_0](s,\theta)\} 
$$
for $(s,\theta) \in \R \times S^2$, where
${\mathcal R}[\psi]$ denotes the Radon transform of $\psi$, i.e.,
$$
{\mathcal R}[\psi](s,\theta)= \int_{y \cdot \theta=s} \psi(y)\,dS_y.
$$
Here $dS_y$ denotes the area element on the plane 
$\{y \in \R^3;\,y \cdot \theta=s\}$.
\begin{lm}\label{Radiation}
Let $\mathcal O$ be non-trapping.
Suppose that $\vec{f}\in \bigl(C^\infty_0(\Omega)\bigr)^2$.
Let $u_0$ be the solution to the mixed problem \eqref{hw}--\eqref{idh}.
Then there exists $\vec{f}_+\in \bigl({\mathcal S}(\R^3)\bigr)^2$
such that we have the following:
For any non-negative integer $k$ and any $\nu>0$,
there exists a positive constant $C$ such that
\begin{align}
\label{1.0}
& \bigl|
u_0(t,x)-r^{-1}{\mathcal F}_0[\vec{f}_+](r-t, \omega)\bigr|_k
\\
&\quad {}+\sum_{a=0}^3 \bigl|
\pa_a u_0(t,x)-\omega_a r^{-1} (\pa_s {\mathcal F}_0[\vec{f}_+]) 
(r-t, \omega)\bigr|_k \nonumber\\
& \qquad \le C(1+t+r)^{-2}(1+|t-r|)^{-\nu}
\text{ for $(t,x)$ with $r\ge t/2\ge 1$,}
\nonumber
\end{align}
where we have put $r=|x|$, $\omega=(\omega_1, \omega_2, \omega_3)=r^{-1}x$,
and $\omega_0=-1$.

In particular, we have
\begin{equation}
\label{f-00}
{\mathcal F}_+(s,\theta)=\pa_s{\mathcal F}_0[\vec{f}_+](s,\theta)
\quad \text{ for $(s,\theta)\in \R\times S^2$},
\end{equation}
where ${\mathcal F}_+$ is given by \eqref{DefRadiation}.
If $\vec{f}$ satisfies
\begin{equation}
\label{SupportCond}
\vec{f}(x)=0 \text{ for $|x|\ge R$}
\end{equation}
with some $R>1$, then
we have
\begin{equation}
\label{SupportR}
{\mathcal F}_0[\vec{f}_+](s,\theta)=\mathcal{F}_+(s,\theta)=0 \quad \text{ for $s\ge R$ and $\theta\in S^2$.}
\end{equation}

Moreover, for any $N>0$ and any multi-index $\gamma$, 
there exists a positive constant $C$ such that
\begin{align}\label{f}
|\Lambda_*^\gamma{\mathcal F}_0[\vec{f}_+](s,\theta)|+|\Lambda_*^\gamma {\mathcal F}_+(s,\theta)|
\le C \jb{s}^{-N}
\quad \text{for}\ \ (s,\theta) \in \R \times S^2.
\end{align}
\end{lm}
\begin{proof}
The existence of $\vec{f}_+$ satisfying \eqref{1.0} follows from
Theorem~1.2 in \cite{kkInverse}, where a stronger estimate than \eqref{1.0}
is obtained.
From \eqref{DefRadiation} and \eqref{1.0}, we obtain
$$
{\mathcal F}_+(s,\theta)=\lim_{t\to\infty} \left.(-t)\pa_tu(t, r\theta)\right|_{r=s+t}=\pa_s{\mathcal F}_0[\vec{f}_+](s,\theta),
$$
which shows \eqref{f-00}.
By the property of finite propagation, \eqref{SupportCond} implies 
$u_0(t,x)=0$ for $|x|\ge t+R$. Hence, multiplying \eqref{1.0} by $r$,
putting $x=(t+s)\theta$ with $s\ge R$ and $\theta\in S^2$, and taking the limit
as $t\to\infty$, we obtain \eqref{SupportR}.

Since it holds
for any $\vec{\varphi} \in ({\mathcal S}(\R^3))^2$, any multi-index $\gamma$,
and any $N>0$ that
$$
|\Lambda_*^\gamma {\mathcal F}_0[\vec{\varphi}](s,\theta)|
\le C \jb{s}^{-N}
\quad \text{for}\ \ (s,\theta) \in \R \times S^2
$$
(see \cite[Lemma 4.2]{kkInverse}), \eqref{f} follows immediately from \eqref{f-00}.
This completes the proof.
\end{proof}
\begin{lm}\label{Triviality}
Let ${\mathcal F}_+(s,\theta)$ and $G$ be given by \eqref{DefRadiation}
and \eqref{a}, respectively. 
We put $A=\sup\{ -2^{-1}G(\theta){\mathcal F}_+(s,\theta); s\in \R,\theta\in S^2\}$.
Suppose that $G\not\equiv 0$ on $S^2$, and $\vec{f}\not\equiv 0$ on $\Omega$.
Then we have $0<A<\infty$. Consequently $\tau_*(=A^{-1})$ 
given by \eqref{DefTau} is a finite positive number.
\end{lm}
\begin{proof}
Let $\vec{f}_+=(f_{0,+}, f_{1,+})\in \bigl({\mathcal S}(\R^3)\bigr)^2$ 
be from Lemma~\ref{Radiation}.
Then, in view of \eqref{f-00}, we have
$$
2\norm{{\mathcal F}_+}{L^2(\R\times S^2)}^2 
=\norm{\nabla_x f_{0, +}}{L^2(\R^3)}^2+\norm{f_{1, +}}{L^2(\R^3)}^2
$$
(see Lax-Phillips~\cite{LaPh}). Since $\vec{f}_+$ is the scattering data for 
the Dirichlet problem \eqref{hw}--\eqref{idh} (see \cite[Theorem 1.1]{kkInverse}), we also have
$$
\norm{\nabla_x f_{0,+}}{L^2(\R^3)}^2+\norm{f_{1,+}}{L^2(\R^3)}^2=\norm{\nabla_x f_0}{L^2(\Omega)}^2+\norm{f_1}{L^2(\Omega)}^2.
$$
Therefore we find that ${\mathcal F}_+\equiv 0$ if and only if $\vec{f}=(f_0,f_1)\equiv 0$.

It is trivial to show $0\le A<\infty$ because of \eqref{f}.
We shall show that if $A=0$, then either $G \equiv 0$ or $\vec{f} \equiv 0$ holds.
We set
$$
X=\{\theta\in S^2;~G(\theta)=0\}, \quad
Y=\{\theta\in S^2;~{\mathcal F}_+(s,\theta)=0 \ \mbox{for all}\ s\in \R \},
$$
which are closed sets because of the continuity of $G(\theta)$ and ${\mathcal F}_+(s,\theta)$.
Now our goal is to show that if $A=0$ then we have either $X=S^2$ or $Y=S^2$
(note that $Y=S^2$ leads to $\vec{f}\equiv 0$);
it suffices to prove that the assumptions $A=0$ and $X\ne S^2$ imply $Y=S^2$.

We assume $A=0$ and $X\ne S^2$ from now on. First we prove that 
\begin{equation}
\label{John-00}
 S^2\setminus X\subset Y.
\end{equation}
Indeed, let ${\theta}^0\in S^2\setminus X$. Since $A=0$ implies $G(\theta){\mathcal F}_+(s,\theta)\ge 0$ for any $(s,\theta)\in \R\times S^2$, and since $G({\theta}^0)\ne 0$,
we find that ${\mathcal F}_+(s,{\theta}^0)$ has the same sign for all $s\in \R$. Hence $\mathcal{F}_0[\vec{f}_+](s, {\theta}^0)$
is (weakly) monotone in $s$ (cf.~\eqref{f-00}), and \eqref{f} implies that $\mathcal{F}_0[\vec{f}_+](s, {\theta}^0)=0$ for all $s\in \R$;
consequently we get $\theta^0\in Y$ in view of \eqref{f-00}.

If $X=\emptyset$, then $Y=S^2$ because of \eqref{John-00}, and we are done. 
Suppose that $X\ne \emptyset$. 
Since $G$ is a polynomial of degree $2$ and $G\not\equiv 0$ on $S^2$,
it is easy to see that $X$ has no interior point; hence, recalling that
$X$ is closed, we get $X=\pa X$. Thus for
any $\theta^0 \in X$ we can take a sequence $\{\theta^{(k)}\}_{k\in \N}\subset S^2\setminus X
(\subset Y)$ such that $\lim_{k\to\infty}\theta^{(k)}=\theta^0$; since $\{\theta^{(k)}\}\subset Y$ and $Y$ is closed, we see that $\theta^0\in Y$. Now we have proved $X\subset Y$, which,
together with \eqref{John-00}, shows $Y=S^2$.
This completes the proof.
\end{proof}

Next we describe basic estimates for the wave equation from \cite{KatKub08MSJ},
and their variant; more general estimates under more general situations are available, however 
we restrict our attention only to the estimates which will be used in this paper.
Recall the definitions \eqref{DefPsi}--\eqref{HomWei}.
\begin{prop}\label{main}
Assume that $\mathcal O$ is non-trapping.
Let $k$ be a non-negative integer.

\noindent
{\rm (1)} 
Let $\rho>0$, and let $u_0$ be the solution to \eqref{hw}--\eqref{idh}.
Suppose that $\vec{f}\in \bigl(C^\infty_0(\Omega)\bigr)^2$. 
Then we have 
\begin{align}\label{ba3}
     & \langle t+|x| \rangle \jb{t-|x|}^\rho 
|u_0(t,x)|_k
\le C {\mathcal A}_{\rho+2,k+3}[\vec{f}\,]
\end{align}
for $(t,x)\in [0,T)\times \overline{\Do}$, where $C$ is a positive
constant depending on $\rho$ and $k$.
\medskip\\
{\rm (2)}
Let $\rho\ge 1$, and let $u$ be the solution to 
\begin{equation}
\label{MP}
\begin{cases}
\dal u(t,x)= h(t,x) & \text{for $(t,x) \in (0,T)\times \Omega$,}\\
u(t,x)=0 & \text{for $(t,x) \in (0,T)\times \partial\Omega$,}\\
\bigl(u(0,x), (\partial_t u)(0,x)\bigr)=\vec{f}(x) & \text{for $x\in \Omega$.}
\end{cases}
\end{equation}
Suppose that $(\vec{f}, h)$ satisfies the compatibility condition
of infinite order.
We also suppose that $h(t,x)=h_1(t,x)+h_2(t,x)$.
Then, for $\mu_1, \mu_2\ge 0$,
we have
\begin{align}\label{ba4}
 & \jbx \langle t -|x| \rangle^{\rho} |\partial u(t,x)|_k
\\ \nonumber
& \quad \le C {\mathcal A}_{\rho+2,k+4}[\vec{f}\,]+C\sum_{j=1}^2 
  \Psi_{\mu_j}(t) \|h_j(t)\!:\!{N_{k+4}(W_{\rho+\mu_j,1-\mu_j})}\|
\end{align}
for $(t,x)\in [0,T)\times \overline{\Do}$.

\end{prop}
The estimate \eqref{ba3} above is just the special case of (4.7) in \cite[Theorem~4.2]{KatKub08MSJ}.
\eqref{ba4} is a variant of (4.8) in \cite[Theorem~4.2]{KatKub08MSJ} and can be proved similarly;
we will give its proof in the Appendix below.
\begin{remark}\label{Two} \normalfont
Here we recall the definition of the compatibility condition:
We say that
$(\vec{f}, h)$ satisfies the compatibility condition of infinite order if
$u^{(j)}(x)\equiv 0$ for any $x\in \pa \Do$ 
and any $j\ge 0$, where $u^{(j)}$ is defined inductively by
$u^{(j)}(x)=\Delta u^{(j-2)}(x)+(\pa_t^{j-2}h)(0,x)$ for $j\ge 2$ with $u^{(j)}(x)=f_j(x)$ for $j=0, 1$.
Observe that we have $(\pa_t^ju_0)(0,x)=u^{(j)}(x)$
for smooth function $u_0$ satisfying \eqref{MP}.
It is easy to see that if $\vec{f}\in \bigl(C^\infty_0(\Do)\bigr)^2$,
then $(\vec{f},0)$ satisfies the compatibility condition
(and it is what we actually need in order to obtain
Lemma~\ref{Radiation} and \eqref{ba3}).
Hence, for 
$\vec{f}\in \bigl(C^\infty_0(\Do)\bigr)^2$,
$(\vec{f},h)$ satisfies the compatibility condition if and only if 
$\bigl(0, h\bigr)$ satisfies the compatibility condition. 

\end{remark}

Next we introduce the well-known elliptic estimate
(for the proof, see, e.g., \cite[Appendix~A]{KatKub08MSJ}).

\begin{lm}\label{elliptic}
Let $\varphi \in H^m(\Do) \cap H_0^1(\Do)$ 
for some integer $m\,(\ge 2)$. 
Then we have
\begin{equation}\label{ap10}
\sum_{|\alpha|=m} \norm{\partial_x^\alpha \varphi}{L^2(\Do)} \le
C(\|\Delta_{x} \varphi\!:\!{H^{m-2}(\Do)}\|+\|\nabla_x \varphi\!:\!{L^2(\Do)}\|).
\end{equation}
\end{lm}

In order to associate decay estimates with the energy estimate,
we use the following variant of the Sobolev type inequality
(for the proof, see, e.g., \cite[Appendix C]{KatKub08MSJ}).

\begin{lm}\label{KlainermanSobolev}\
There exists a positive constant $C$
such that 
\begin{equation}\label{ap21}
\sup_{x \in \Do}\, \jb{x} |\varphi(x)|
\le C \sum_{|\alpha| \le 2} 
\norm{Z^\alpha \varphi}{L^2(\Do)}
\end{equation}
for any $C^2(\overline{\Do})$-function $\varphi$
vanishing outside some bounded set.
\end{lm}

\section{Approximate solutions}
\label{Approx}
This section is the core of the present
paper, namely, we shall construct an approximate solution
(see Proposition \ref{Lemma3.3} below).


Let $u_0$ be the solution to \eqref{hw}--\eqref{idh}.
Throughout this section we assume that \eqref{SupportCond} holds
for some $R>1$ (recall that we have assumed ${\mathcal O}\subset B_1(0)$).
Let ${\mathcal F}_+(s,\theta)$
be defined by \eqref{DefRadiation}, and 
let $P(s,\theta,\tau)$ be the unique solution of
\begin{align}\label{3.5}
& 2\partial_\tau P(s,\theta, \tau)=
    -G(\theta) P^2(s,\theta,\tau)
    \quad \text{for $s \in {\R}$, $\theta \in S^2$, $0 \le \tau < \tau_*$},
\\ \label{3.6}
& 
P(s,\theta,0)={\mathcal F}_+(s,\theta) \quad \text{for $s\in \R$, $\theta\in S^2$},
\end{align}
where $G(\theta)$ and $\tau_*$ are defined by \eqref{a} and \eqref{DefTau}, respectively.
Then we have
\begin{equation}
\label{Expression01}
P(s,\theta, \tau)=\frac{{\mathcal F}_+(s, \theta)}{1+2^{-1}G(\theta)
{\mathcal F}_+(s,\theta)\,\tau}
\end{equation}
for $s\in \R$, $\theta\in S^2$, and $0\le \tau<\tau_*$.
Observe that we have $1+2^{-1}G(\theta){\mathcal F}_+(s,\theta)\tau\ge 1-\tau/\tau_*$ 
for $0\le \tau<\tau_*$.
We define $p$ by
\begin{equation}
\label{Expression02}
p(s, \theta, \tau)
=-\int_{s}^\infty \frac{{\mathcal F}_+(s', \theta)}{1+2^{-1}G(\theta)
{\mathcal F}_+(s',\theta)\,\tau}ds'
\end{equation}
for $s\in \R$, $\theta\in S^2$, and $0\le \tau<\tau_*$,
so that we have 
\begin{equation}
\label{DerivP}
\pa_s p(s,\theta, \tau)=P(s,\theta, \tau).
\end{equation}
Note that the right-hand side of \eqref{Expression02} is finite because of
\eqref{f} with $N>1$. From \eqref{SupportR} we also get
\begin{equation}
\label{SupportP}
p(s,\theta,\tau)=\pa_s p(s,\theta,\tau)=0\quad \text{ for $s\ge R$.}
\end{equation}
By \eqref{Expression02}, \eqref{f-00}, and \eqref{SupportR}, we obtain
\begin{equation}
p(s,\theta,0)={\mathcal F}_0[\vec{f}_+](s,\theta)\quad \text{ for $(s,\theta)\in \R\times S^2$,}
\end{equation}
where $\vec{f}_+$ is from Lemma~\ref{Radiation}.
\begin{lm}\label{Lemma3.2}
Assume that \eqref{SupportCond} holds for some $R>1$.
Let $0<\tau_0<\tau_*$. 
Then 
for any $N>0$, and for any multi-indices 
$\beta=(\beta_0,\dots,\beta_4)$ and $\gamma=(\gamma_0, \dots, \gamma_3)$,
there exists a positive constant $C=C(\tau_0, \beta, \gamma, N)$ such that
\begin{align}
& | \Lambda^\beta p(s,\theta,\tau)| \le C,
\label{3.9-o}\\
& | \Lambda^\beta \pa_s p(s,\theta,\tau)| \le C\jb{s}^{-N},
\label{3.9}\\
& | \Lambda_*^\gamma \{ p(s,\theta,\tau)-{\mathcal F}_0[\vec{f}_+](s,\theta)
\}|\le C \tau ,
 \label{3.10-o}\\
& | \Lambda_*^\gamma \{ \pa_s p(s,\theta,\tau)-{\mathcal F}_+(s,\theta)
                     \}|
 \le C \tau\jb{s}^{-N}
 \label{3.10} 
\end{align}
for all $(s,\theta,\tau) \in {\R} \times S^2 \times [0,\tau_0]$.
\end{lm}
\begin{proof}
Noting that $1+2^{-1}G(\theta) {\mathcal F}_+(s,\theta)\,\tau \ge 1-(\tau_0/\tau_*)$
for $s \in \R$, $\theta \in S^2$ and $0 \le \tau \le \tau_0$, we get (\ref{3.9})
from \eqref{f}, \eqref{Expression01}, and \eqref{DerivP}.
\eqref{3.9-o} follows immediately from \eqref{3.9} with $N>1$ because of
\eqref{SupportP}.

From \eqref{DerivP} and \eqref{Expression01} we get
\begin{align*}
\pa_s p(s,\theta,\tau)-{\mathcal F}_+(s,\theta)
=& -\frac{2^{-1}G(\theta){\mathcal F}_+^2(s,\theta)\tau}{1+2^{-1}G(\theta){\mathcal F}_+(s,\theta)\tau}.
\end{align*}
Recalling \eqref{f-00} and \eqref{SupportR}, and integrating the above identity, we obtain
\begin{align*}
p(s,\theta,\tau)-{\mathcal F}_0[\vec{f}_+](s,\theta)
=& \int_s^\infty \frac{2^{-1}G(\theta){\mathcal F}_+^2(s',\theta)\tau}{1+2^{-1}G(\theta){\mathcal F}_+(s',\theta)\tau}ds'.
\end{align*}
Hence (\ref{f}) yields \eqref{3.10-o} and \eqref{3.10}. 
This completes the proof.
\end{proof}

For a function $\varphi=\varphi(s,\theta, \tau)$, we define
$$
\widetilde{\varphi}(t,x):=\varphi(r-t, \omega, \ve \log t)
$$
with $r=|x|$ and $\omega=r^{-1}x$.
Then we have
\begin{align}\label{3.15+o}
& \partial_t \widetilde{\varphi}=-\widetilde{\pa_s \varphi}+\varepsilon t^{-1}
 \widetilde{\pa_\tau \varphi}, \quad
  O_{ij} \widetilde{\varphi}=\widetilde{o_{ij} \varphi} \quad (1\le i, j \le 3),
\\ \label{3.15++o}
& \nabla_{x} \widetilde{\varphi}= \omega \widetilde{\pa_s \varphi}-r^{-1} 
\omega \wedge \widetilde{o\varphi}
\quad\text{with}\ o\varphi=(o_{23}\varphi, o_{31}\varphi, o_{12}\varphi),
\end{align}
where we have used \eqref{rew1} to obtain \eqref{3.15++o}.

For $p(s,\theta, \tau)$ defined by \eqref{Expression02} we set 
\begin{align}\label{AppCone}
 w(t,x)=\ve |x|^{-1}\,\widetilde{p}(t,x) 
\end{align}
for $(t,x) \in \bigl[1,\exp(\tau_*/\ve)\bigr) \times (\R^3\setminus\{0\})$.
Then it will serve as an approximation of $u$ for large $t$,
and the estimates obtained in the above lemma are transfered as follows.

\begin{cor}\label{cor4.2}
We assume 
that \eqref{SupportCond} holds for some $R>1$.
Let $0<\tau_0<\tau_*$, and $0<\ve\le 1$. 
Then for 
any non-negative integer $k$, there
exists a positive constant $C=C(\tau_0, k)$ such that
\begin{align}
\label{3.17-o}
|w(t,x)|_k \le & C \ve \jb{t+|x|}^{-1}, \\
 \label{3.17}
|\pa w(t,x)|_k \le & C \ve \jb{t+|x|}^{-1}\jb{t-|x|}^{-1}
\end{align}
for $t/2 \le |x| \le t+R$ and $1 \le t \le \exp(\tau_0/\ve)$.
Moreover we have
\begin{align}
|w(t,x)-\ve u_0(t,x)|_k
 \le & C \ve \left(\log \frac{2}{\ve}\right) \jb{t+|x|}^{-2},
 \label{3.18-o}\\
|\pa_t \{w(t,x)-\ve u_0(t,x)\}|_k
 \le & C \ve \left(\log \frac{2}{\ve}\right) \jb{t+|x|}^{-2} \jb{t-|x|}^{-1},
 \label{3.18}
\end{align}
for $t/2 \le |x| \le t+R$ and $2 \le t \le 2/\ve$.
\end{cor}
\begin{proof}
We write $x=r\omega$ with $r=|x|$ and $\omega\in S^2$.
We suppose that $t/2\le r \le t+R$ and $1\le t\le \exp(\tau_0/\ve)$ in what follows.
Then we have 
\begin{equation}
\label{Equi00}
|t^{-1}|_k+|r^{-1}|_k+|\jb{t+r}^{-1}|_k\le C\jb{t+r}^{-1}.
\end{equation}
From \eqref{3.9-o}, \eqref{3.9} (with $N=1$), \eqref{3.15+o}, and \eqref{3.15++o},
we get
\begin{align}
 |\widetilde{p}(t,x)|_k \le & C \sum_{|\beta|\le k} \bigl|\widetilde{\Lambda^\beta p}(t,x)\bigr| 
\le C \label{3.15},\\
 |\pa \widetilde{p}(t,x)|_k \le & C \sum_{|\beta|\le k} 
    \bigl|\widetilde{\Lambda^\beta \pa_s p}(t,x)\bigr| 
 {}+C\jb{t+r}^{-1}\sum_{|\beta|\le k+1} \bigl|\widetilde{\Lambda^\beta p}(t,x)\bigr|
 \label{3.15-o}\\
 \le & C\jb{t-r}^{-1}.
 \nonumber
\end{align}
From \eqref{3.15} and \eqref{3.15-o},
we obtain \eqref{3.17-o} and \eqref{3.17}, because of \eqref{Equi00}.

Next we prove (\ref{3.18}). 
Suppose that we also have $2\le t\le 2\ve^{-1}$ from now on. 
Then we have
$\ve\le C\jb{t+r}^{-1}$.
By \eqref{3.15+o} we get
\begin{align}
\pa_t \widetilde{p}(t,x)+{\mathcal F}_+(r-t,\omega)
=&-\bigl\{\widetilde{\pa_s p}(t,x)-{\mathcal F}_+(r-t,\omega)\bigr\}\\
&{}+\ve t^{-1}\widetilde{\pa_\tau p}(t,x).
\nonumber
\end{align}
We can inductively obtain the explicit formula for $\pa_t^m\bigl(\pa_t\widetilde{p}+{\mathcal F}_+\bigr)$ for $m\ge 1$.
Therefore, it follows from \eqref{3.9-o}, 
\eqref{3.10} with $N=1$, \eqref{3.15+o}, and \eqref{3.15++o} that 
\begin{align}
\label{kk01}
& |\pa_t \widetilde{p}(t,x)+{{\mathcal F}_+}(r-t,\omega)|_k  \\
& \quad
  \le C\sum_{|\gamma|\le k}
  |\widetilde{\Lambda_*^\gamma \pa_s p}(t,x)-(\Lambda_*^\gamma{\mathcal F}_+)(r-t,\omega)|
  \nonumber\\
  & \qquad
  {}+C\ve\jb{t+r}^{-1}\sum_{|\beta|\le k}|\widetilde{\Lambda^\beta\pa_\tau p}(t,x)|
\nonumber\\
\nonumber
& \quad \le C \ve \left((\log t)\jb{t-r}^{-1}+\jb{t+r}^{-1}\right)
\\ \nonumber
& \quad \le C \left(\log \frac2{\ve}\right) \jb{t+r}^{-1} \jb{t-r}^{-1}.
\end{align}
Besides, (\ref{1.0}) yields 
\begin{align}\label{3.20+}
|r^{-1} {\mathcal F}_+(r-t,\omega)+\pa_t u_0(t,x)|_k 
   \le C \jb{t+r}^{-2} \jb{t-r}^{-1}.
\end{align}
From \eqref{kk01} and \eqref{3.20+} we obtain \eqref{3.18}.

Similarly  \eqref{3.18-o} follows from \eqref{3.9-o}, \eqref{3.10-o}, \eqref{3.15+o}, and \eqref{3.15++o}
with the help of \eqref{1.0}. 
This completes the proof.
\end{proof}

Now we are in a position to construct an approximate solution $u_1$:
Let ${\chi}$ and ${\xi}$ be smooth and non-negative functions on $[0,\infty)$
such that
\begin{align*}
{\chi}(s)=\begin{cases} 1, & s\le 1,\\
                        0, & s\ge 2,\\
          \end{cases}
\quad \
{\xi}(s)=\begin{cases} 0, & s\le 1/2,\\
                               1, & s\ge 3/4.
               \end{cases}
\end{align*}
We put $\chi_{\ve}(t)=\chi(\varepsilon t)$ for $\ve>0$,
and  $\eta(t,x)=\xi(|x|/t)$.
Let $u_0$ be the solution of the mixed problem (\ref{hw})--(\ref{idh}),
and let $w$ be given by \eqref{AppCone}. 
We define the approximate solution $u_1$ by
\begin{align}\label{3.4}
 u_1(t,x)=\chi_\ve(t)\ve u_0(t,x)+\bigl(1-\chi_{\ve}(t)\bigr) 
 \eta(t,x) w(t,x)
\end{align}
for $(t,x) \in \bigl[0,\exp(\tau_*/\ve)\bigr) \times \overline{\Omega}$.
By \eqref{SupportCond} and the property of finite propagation, we have
$|x|\le t+R$ in $\supp u_0$. Hence, recalling \eqref{SupportP}, we find that
\begin{equation}
\label{SupportA}
u_0(t,x)=w(t,x)=u_1(t,x)=0\quad \text{ for $|x|\ge t+R$}.
\end{equation}
We are going to evaluate $u_1$ and the error term
\begin{equation}\label{3.3}
E(u_1)(t,x)=\dal u_1(t,x)-F\bigl(\partial u_1(t,x)\bigr)
\end{equation}
that is expected to be small if $u_1$ is a good approximate solution to
\eqref{eq}.
\begin{prop}
\label{Lemma3.3}
We assume that \eqref{SupportCond} holds for some $R>1$.
Let $0<\tau_0<\tau_*$, $k$ be a non-negative integer, 
$0\le\lambda\le 1/2$, $0<\mu\le 1/4$, and $0<\ve \le 1/2$.
Then there exists a positive constant $C=C(\tau_0, k, \lambda, \mu)$ such that 
\begin{align}
& |u_1(t,x)|_k \le C \ve \jb{t+|x|}^{-1},
  \label{3.11}\\
& |\pa u_1(t,x)|_k \le C \ve \jb{t+|x|}^{-1} \jb{t-|x|}^{-1},
  \label{3.11-o}\\
& |E(u_1)(t,x)|_k
 \le C \ve^{1+\lambda} \jb{t+|x|}^{-2+\lambda-\mu} \jb{t-|x|}^{-1+\mu}
 \label{3.11+}
\end{align}
for $(t,x) \in [0, \exp(\tau_0/\ve)] \times \overline{\Omega}$,
and 
\begin{equation}
\bigl\| |E(u_1)(t,\cdot)|_k \!:\! L^2(\Omega)\bigr\|
\le C\ve^{1+\lambda} (1+t)^{-(3/2)+\lambda}
  \label{3.11++}
\end{equation}
for $t\in [0, \exp(\tau_0/\ve)]$.
\end{prop}

\begin{proof}
We write $x=r\omega$ with $r=|x|$ and $\omega\in S^2$.
First of all, we derive the estimates for the cut-off functions.
Since we have $\ve t\le 2$ for $t\in \supp \chi_\ve$ and we have assumed
$0<\ve\le 1/2$, we get
\begin{equation}
\label{CutOffEst01}
\ve\le C(1+t)^{-1} \quad \text{for $t\in \supp \chi_\ve$}.
\end{equation}
Let $m$ be a non-negative integer.
Since $\chi_\ve(t)=\chi(\ve t)$, \eqref{CutOffEst01} leads to
\begin{equation}
\label{CutOffEst12}
\left|\frac{d^m\chi_\ve}{dt^m}(t)\right| =
\ve^m \left|\frac{d^m\chi}{dt^m}(\ve t)\right|\le C_m (1+t)^{-m}
\quad \text{for $t\ge 0$},
\end{equation}
where $C_m$ is a positive constant.
We turn our attention to $\eta$; observing that $t\ge \ve^{-1}\ge 2$ for $t\in \supp(1-\chi_\ve)$,
we only need the estimate of $\eta$ for $t\ge 2$.
Since $\eta$ is depending only on $r/t$
and we have $t/2\le r \le 3t/4$ for $(t,x)\in \supp \pa \eta$,
it is easy to see that
$O_{jk}\eta(t,x)=0$ for $1\le j, k\le 3$, and that
\begin{equation}
\label{CutOffEst13}
\sum_{|\alpha|=m} |\pa^\alpha \eta(t,x)| \le C_m \jb{t+r}^{-m} \quad \text{for $(t,x)\in [2,\infty)\times \R^3$},
\end{equation}
where $m$ is a non-negative integer, $\pa=(\pa_t,\nabla_x)$, and $\alpha$ is a multi-index.
Besides, since $r\le 3t/4$ in $\supp(1-\eta)\cup\supp \pa\eta$, we get
\begin{equation}
\label{CutOffEst14}
\jb{t-r}^{-1}\le C\jb{t+r}^{-1}\quad\text{for $(t,x)\in \supp (1-\eta) \cup \supp \pa \eta$}.
\end{equation}

Now we are going to prove \eqref{3.11} and \eqref{3.11-o}.
It follows from \eqref{ba3} with $\rho=2$ that
\begin{align}
 |u_0(t,x)|_{k+1} \le C {\mathcal A}_{4, k+4}[\vec{f}\,]\,
 \jb{t+|x|}^{-1} \jb{t-|x|}^{-2}
 \label{3.14}
\end{align}
for $(t,x) \in [0,\infty) \times \overline{\Omega}$.
Recalling \eqref{SupportA}, we have $(1+t)^{-1}\le C \jb{t+r}^{-1}$ for $(t,x)\in \supp w$.  
Therefore, we get \eqref{3.11} and \eqref{3.11-o} from \eqref{3.17-o}, \eqref{3.17},
\eqref{CutOffEst12}, \eqref{CutOffEst13}, and \eqref{3.14}.

Next we consider \eqref{3.11+} and \eqref{3.11++}.
If we set
\begin{align}\label{DefV}
v(t,x)=\eta(t,x) w(t,x)-\ve u_0(t,x),
\end{align}
then we have $u_1=\ve u_0+(1-\chi_\ve) v$ by \eqref{3.4}.
It follows that
\begin{equation}
E(u_1)=\dal u_1-F(\pa u_1)=I_0+I_1+I_2+I_3,
\end{equation}
where 
$$
\begin{array}{ll}
I_0 = -\chi_\ve(t)F(\pa u_1), 
& I_1 =-\chi_\ve''(t) v(t,x),\\
I_2 = -2\chi_\ve'(t) \pa_t v(t,x), &
I_3 = \bigl(1-\chi_\ve(t)\bigr) \left\{\dal \bigl( \eta(t,x) w(t,x)\bigr)-F(\pa u_1)\right\}.
\end{array}
$$
We will estimate $I_j$ for $0\le j\le 3$.

By \eqref{CutOffEst01} and \eqref{SupportA}, we have
\begin{align}
\label{basic00}
\ve \le C\jb{t+r}^{-1}\quad \text{for $(t,x)\in 
\supp I_0\cup \supp I_1\cup \supp I_2$}.
\end{align}
Let 
$0\le \lambda \le 1/2$ and $0<\mu \le 1/4$.
From \eqref{3.11-o} we get
\begin{align}
\label{basic02}
|I_0|_k\le & C|F(\pa u_1)|_k\le C\ve^2 \jb{t+r}^{-2}\jb{t-r}^{-2}\\
       \le & C\ve^{1+\lambda}\jb{t+r}^{-3+\lambda}\jb{t-r}^{-2} 
\nonumber \\
       \le & C \ve^{1+\lambda}\jb{t+r}^{-2+\lambda-\mu}\jb{t-r}^{-1+\mu},
\nonumber
\end{align}
where we have used \eqref{basic00} to obtain the second line.
From the second line of \eqref{basic02}, we also get
\begin{equation}
\label{basic021}
\bigl\||I_0|_k\bigr\|_{L^2}\le C\ve^{1+\lambda}\jb{t}^{-2+\lambda}
\le C\ve^{1+\lambda}\jb{t}^{-(3/2)+\lambda}.
\end{equation}

Fix $\delta\in (0,1/4]$.
Then we get
\begin{equation}
\label{LogBasic}
\log(2/\ve)\le C_\delta \ve^{-\delta}\text{ for $0<\ve\le 1/2$,}
\end{equation}
where $C_\delta$ is a constant depending only on $\delta$.
 Observe that \eqref{CutOffEst13}, \eqref{CutOffEst14} and \eqref{3.14} yield
\begin{equation}
\left|\bigl(1-\eta(t,x)\bigr)u_0(t,x)\right|_k+\left|\bigl(1-\eta(t,x)\bigr)\pa_t 
u_0(t,x)\right|_k
\le C\jb{t+r}^{-3}.
\label{CutOffEst41}
\end{equation}
Observe also that we have $2\le t\le 2/\ve$ and $t/2\le r\le t+R$ in 
$\supp \bigl(\chi_\ve'' \eta(w-\ve u_0)\bigr)$.   
Therefore, writing $$
I_1=-\ve^2 \chi''(\ve t)\bigl(\eta(w-\ve u_0)-\ve(1-\eta)u_0\bigr),
$$
by \eqref{3.18-o}, \eqref{CutOffEst41}, \eqref{basic00} and \eqref{LogBasic},
we get
\begin{align}
\label{basic03}
|I_1|_k\le & C\ve^3 \left(\log\frac{2}{\ve}\right)\jb{t+r}^{-2}
\le C\ve^{2-\delta} \jb{t+r}^{-2}\jb{t-r}^{-1}.
\end{align}
Similarly, if we write 
\begin{align*}
I_2=& -2\ve\chi'(\ve t)\bigl(
(\pa_t\eta)(w-\ve u_0)+(\pa_t \eta)\ve u_0\\
& \qquad\qquad\qquad {}+\eta(\pa_t w-\ve \pa_t u_0)-(1-\eta)\ve \pa_t u_0
\bigr),
\end{align*} 
then it follows from 
\eqref{3.18-o}, \eqref{3.18}, 
\eqref{CutOffEst13}, 
\eqref{3.14},
\eqref{LogBasic}, and \eqref{CutOffEst41}
that
\begin{align}
\label{basic04}
|I_2|_k \le & C\ve^2 \left(\log\frac{2}{\ve}\right)\jb{t+r}^{-2}\jb{t-r}^{-1}
\\
\le & 
C\ve^{2-\delta}\jb{t+r}^{-2}\jb{t-r}^{-1}.
\nonumber
\end{align}
By \eqref{basic03}, \eqref{basic04}, and \eqref{basic00} we get
\begin{align}
\label{Con01}
|I_1+I_2|_k\le & C\ve^{1+\lambda}\jb{t+r}^{-3+\lambda+\delta}\jb{t-r}^{-1}\\
\le & C\ve^{1+\lambda}\jb{t+r}^{-2+\lambda-\mu}
\jb{t-r}^{-1+\mu},
\nonumber
\end{align}
since $\mu+\delta<1$.
Moreover, integrating  
the first line of \eqref{Con01},
we get
\begin{align}
\label{Con11}
\bigl\||I_1+I_2|_k\bigr\|_{L^2} \le & 
C\ve^{1+\lambda}\jb{t}^{-2+\lambda+\delta}
\le C\ve^{1+\lambda}\jb{t}^{-(3/2)+\lambda}.
\end{align}

We further decompose $I_3$ as
\begin{equation}
I_3=\bigl(1-\chi_\ve(t)\bigr)(I_{31}+I_{32}+I_{33}+I_{34}),
\end{equation}
where 
\begin{align*}
I_{31}=& -F(\pa u_1)+\eta F(\pa w),\quad  I_{32}=(\dal \eta)w,\\
I_{33}=& 2\bigl\{
(\pa_t \eta)(\pa_t w)-(\nabla_x \eta)\cdot(\nabla_x w) \bigr\}, \quad I_{34}=\eta \bigl(\dal 
 w-F(\pa w)\bigr).
\end{align*}
We start with the estimate for $I_{31}$: If $0\le \ve t\le 2$, then similarly to \eqref{basic02} we get
$$
|(1-\chi_\ve)I_{31}|_k\le C\ve^{1+\lambda}\jb{t+r}^{-3+\lambda}\jb{t-r}^{-2},
$$
by \eqref{3.17} and \eqref{3.11-o}.
If $r\le 3t/4$, then we have $\jb{t-r}^{-1}\le C\jb{t+r}^{-1}$, and we get 
$$
|(1-\chi_\ve)I_{31}|_k\le C\ve^2 \jb{t+r}^{-4}.
$$
Since $\eta=1$ and 
$u_1=w$ when $\ve t\ge 2$ and $r\ge 3t/4$, we get $I_{31}=0$.
Summing up, we have proved
\begin{equation}
\label{basic07}
|(1-\chi_\ve)I_{31}|_k\le 
C\ve^{1+\lambda}\jb{t+r}^{-3+\lambda}\jb{t-r}^{-2}
{}+C\ve^{2}\jb{t+r}^{-4}.
\end{equation}

Before we proceed further, we note that 
\begin{equation}
\label{basicSupp}
1\le \ve^{-1}/2\le t/2\le r\le t+R
\text{ in $\supp \bigl(I_3-(1-\chi_\ve) I_{31}\bigr)$}.
\end{equation}
From \eqref{3.17-o} and \eqref{CutOffEst13} (with $m=2$) we get
\begin{equation}
\label{basic05}
|(1-\chi_\ve)I_{32}|_k\le C\ve \jb{t+r}^{-3}. 
\end{equation}
Using \eqref{CutOffEst13} (with $m=1$) and \eqref{CutOffEst14},
from \eqref{3.17} we get
\begin{equation}
\label{basic06}
|(1-\chi_\ve)I_{33}|_k\le C\ve \jb{t+r}^{-3}.
\end{equation}

We are going to estimate $I_{34}$ by writing
\begin{align*}
\dal w-F(\pa w)= & \left(\dal w+\frac{2\ve^2}{tr}\widetilde{\pa_\tau\pa_sp}\right)
{}-\left(\frac{2\ve^2}{tr}\widetilde{\pa_\tau\pa_sp}+\frac{\ve^2}{r^2}G(\omega)
\bigl(\widetilde{\pa_s p}\bigr)^2\right)\\
&+\left(
\frac{\ve^2}{r^2}G(\omega)
\bigl(\widetilde{\pa_s p}\bigr)^2-F(\pa w)\right).
\end{align*}
We assume that $(t,x)\in \supp (1-\chi_\ve) I_{34}$ in the following estimates
\eqref{basic081stat}--\eqref{basic082},
so that $t$ and $r$ are equivalent to $\jb{t+r}$ (see \eqref{basicSupp}).
Recalling \eqref{AppCone},  \eqref{3.15+o}, and \eqref{3.15++o}, we get
\begin{align}
\label{basic081stat}
\dal w=&
\ve r^{-1} \bigl((\pa_t^2-\pa_r^2)\widetilde{p}-r^{-2}\Delta_\omega\widetilde{p}\bigr)\\
=& -\frac{2\ve^2}{tr}\widetilde{\pa_\tau\pa_s p}+\frac{\ve^2}{t^2r}(\ve\widetilde{\pa_\tau^2 p}-\widetilde{\pa_\tau p})-\frac{\ve}{r^3}\widetilde{\Delta_\theta p},
\nonumber
\end{align}
where $\Delta_\omega=\sum_{1\le j<k\le 3} O_{jk}^2$, and $\Delta_\theta=\sum_{1\le j<k\le 3}
o_{jk}^2$.
Hence \eqref{3.9-o} yields
\begin{equation}
\label{basic081}
\left|\dal w+\frac{2\ve^2}{tr}\widetilde{\pa_\tau\pa_s p}\right|_k\le C\ve \jb{t+r}^{-3}.
\end{equation}
Using \eqref{3.15+o} and \eqref{3.15++o} for $\varphi=\pa_s p$,
from \eqref{3.9-o} and \eqref{3.9} (with $N=1$) we get
\begin{equation}
\label{4.9'}
|\widetilde{\pa_s p}(t,x)|_k\le C \jb{t-r}^{-1}.
\end{equation}
By \eqref{3.5}, \eqref{DerivP}, and \eqref{4.9'}, we obtain
\begin{align}
\label{basic083}
\left|\frac{2\ve^2}{tr}\widetilde{\pa_\tau\pa_s p}+\frac{\ve^2}{r^{2}} G(\omega)\bigl(\widetilde{\pa_s p}\bigr)^2\right|_k=&
\left|\frac{(t-r)\ve^2}{tr^{2}} G(\omega)\bigl(\widetilde{\pa_s p}\bigr)^2\right|_k\\
 \le & C\ve^2 \jb{t+r}^{-3}\jb{t-r}^{-1}.   \nonumber
\end{align}
It follows from \eqref{3.9-o}, \eqref{3.15+o}, and \eqref{3.15++o} 
that
\begin{equation}
\label{ApproxF}
\sum_{a=0}^3 |\ve r^{-1}\omega_a\widetilde{\pa_s p}-\pa_a w|_k\le C\ve \jb{t+r}^{-2},
\end{equation}
where $\omega_0=-1$.
If we put 
$$
Q_{a,b}=(\ve r^{-1}\omega_a \widetilde{\pa_s p})(\ve r^{-1}\omega_b \widetilde{\pa_s p})-(\pa_a w)(\pa_b w),
$$
then, by \eqref{4.9'} and \eqref{ApproxF}, we obtain
\begin{equation}
|Q_{a,b}|_k\le C\ve^2 \jb{t+r}^{-3}\jb{t-r}^{-1},
\end{equation}
which immediately leads to
\begin{align}
\label{basic082}
\left|\frac{\ve^2}{r^{2}} G(\omega)\bigl(\widetilde{\pa_s p}\bigr)^2
{}-F(\pa w)\right|_k=&
\left|\sum_{a,b}g^{a,b}Q_{a,b}\right|_k\\
\le & C\ve^2 \jb{t+r}^{-3}\jb{t-r}^{-1}. 
\nonumber
\end{align}
Now \eqref{basic081}, \eqref{basic083}, and \eqref{basic082} lead to
\begin{equation}
\label{basic08}
|(1-\chi_\ve)I_{34}|_k\le C\ve \jb{t+r}^{-3}.
\end{equation}

From \eqref{basic07}, \eqref{basic05}, \eqref{basic06}, and \eqref{basic08},
we obtain
\begin{align}
\label{Con02}
|I_3|_k\le C\ve \jb{t+r}^{-3}+C\ve^{1+\lambda}\jb{t+r}^{-3+\lambda}\jb{t-r}^{-2}.
\end{align}
Since $\ve t\ge 1$ in $\supp I_3$, we have
\begin{equation}
\label{basic01}
\ve \ge t^{-1}\ge \jb{t+r}^{-1}\quad \text{for $(t,x)\in \supp I_3$}.
\end{equation}
If we use \eqref{basic01}
to rewrite the first term on the right-hand side of \eqref{Con02}, 
then we get
\begin{align}
\label{Con03a}
|I_3|_k\le C\ve^{1+\lambda} \jb{t+r}^{-3+\lambda},
\end{align}
which yields
\begin{align}
\label{Con03}
|I_3|_k\le C\ve^{1+\lambda} \jb{t+r}^{-2+\lambda-\mu}\jb{t-r}^{-1+\mu}
\end{align}
and
\begin{align}
\label{Con12}
\bigl\||I_{3}|_k\bigr\|_{L^2} \le & C\ve^{1+\lambda}
\left(\int_0^\infty\jb{t+r}^{-6+2\lambda}r^2dr\right)^{1/2}\\
\le & C\ve^{1+\lambda}(1+t)^{-(3/2)+\lambda}.
\nonumber
\end{align}

Finally \eqref{3.11+} follows from \eqref{basic02}, \eqref{Con01}, and 
\eqref{Con03}.
We also obtain \eqref{3.11++} from \eqref{basic021}, \eqref{Con11}, and \eqref{Con12}.
This completes the proof.
\end{proof}

\section{Proof of Theorem \ref{Th.1.1}}
\label{proof}

Let $\vec{f}\in \bigl(C^\infty_0(\Omega)\bigr)^2$.
Then we have \eqref{SupportCond} for some $R>1$.
By the local existence theorem (see, e.g., \cite{ShiTsu86})
there exists a smooth solution $u$ to \eqref{eq}--\eqref{id}
for some $T>0$.
We write $T_\ve$ for its lifespan.
Let $u_0$ be the solution to \eqref{hw}--\eqref{idh}, and 
let $u_1$ be defined by \eqref{3.4}. 
Suppose that $0<\ve \le 1/2$.
Since we have assumed ${\mathcal O}\subset B_1(0)$,
taking the support of 
$(1-\chi_\ve)\eta$ into account,
we have $u_1(t,x)=0$ for $(t,x)\in [0, T_\ve)\times \pa\Omega$.
By the definition, we also have $u_1(t,x)=\ve u_0(t,x)$ for $0\le t\le \ve^{-1}$
and $x\in \overline{\Omega}$.
Therefore, setting
$$
u_2(t,x)=u(t,x)-u_1(t,x), 
$$
we find
\begin{align}
\label{Req}
& \dal u_2=H(u_1,u_2)-E(u_1) & \text{ in $[0,T_\ve)\times \Omega$},\\
\label{Rbc}
& u_2(t,x)=0 & \text{ for $(t,x)\in [0, T_\ve)\times \pa\Omega$,}\\
\label{Rid}
& u_2(0,x)=\pa_t u_2(0,x)=0 & \text{ for $x\in \Omega$,}
\end{align}
where $E(u_1)$ is defined by \eqref{3.3}, and $H(u_1, u_2)$ is given by
$$
H(u_1,u_2)=F(\pa u_1+\pa u_2)-F(\pa u_1).
$$
For any non-negative integer $k$, there exists a constant $C_k$ such that
\begin{equation} \label{Initialu2}
\sup_{x\in \overline{\Omega}}|u_2(0,x)|_k\le C_k \ve^2,
\end{equation}
because $\dal u_2=F(\pa u)$ for $0\le t\le \ve^{-1}$, and $u_2(0,x)=\pa_t u_2(0,x)=0$.
Note that we have $u_2(t,x)=0$ for $|x|\ge t+R$.

First of all, we show that Theorem \ref{Th.1.1} is a consequence of
the following lemma.
\begin{lm}\label{Lemma5.1}
Let $K$ be a fixed integer with $K \ge 19$.
Then for $\tau_0\in (0, \tau_*)$
there exists a pair of positive numbers 
$(M,\ve_0)=\bigl(M(K,\tau_0), \ve_0(K,\tau_0)\bigr)$
such that
the following assertion is valid:
If $0<\ve\le \ve_0$, $0<T<\min\{T_\ve, \exp(\tau_0/\ve)\}$,
and
\begin{align}\label{5.1}
\sup_{(t,x)\in [0, T] \times\overline{\Omega}} \jbx \jb{t-|x|}|\pa u_2(t,x)|_{K} \le M\ve,
\end{align}
then we have
\begin{align}\label{5.2}
\sup_{(t,x)\in [0, T] \times\overline{\Omega}} \jbx \jb{t-|x|}|\pa u_2(t,x)|_{K} \le \frac{M\ve}{2}.
\end{align}
\end{lm}
\noindent {\it Proof of Theorem~\ref{Th.1.1}.}
We fix $\tau_0 \in (0,\tau_*)$ and an integer $K\ge 19$.
We define
$$
A_K(T)=\sup_{(t,x)\in[0,T]\times \overline{\Omega}}\, \jb{x} \jb{t-|x|} |\pa u_2(t,x)|_K
$$
for $0\le T<T_\ve$.
Note that the boundedness of $A_K$  implies that of $|\pa u|_K$,
because we have the estimate \eqref{3.11-o} for $|\pa u_1|_K$.
Taking the support of $u_2$ into account, from \eqref{Initialu2} we get
\begin{align*}
A_K(0)=\sup_{x\in \overline{\Omega}}\, \jb{x}^2 |(\pa u_2)(0,x)|_K \le C_0 \ve^2
\end{align*}
with a positive number $C_0$. 

Suppose that 
$$
0<\ve<\min\{\ve_0, M/(2C_0)\},
$$
where $M=M(K,\tau_0)$ and $\ve_0=\ve_0(K,\tau_0)$ are 
from Lemma~\ref{Lemma5.1}.
Then we have $A_K(0)\le M\ve/2$, and by the continuity of $A_K$ 
there exists a positive number $T^0$ such that
$A_K(T^0)\le M\ve$.
We define
$$
T_\ve^*=\sup\{T\in [0, T_\ve)\,;\, A_K(T)\le M \ve \}.
$$
Note that we have $T_\ve^*\ge T^0>0$. Now we suppose that 
\begin{equation}
\label{MM}
T_\ve<\exp (\tau_0/\ve).
\end{equation}
Then we get $T_\ve^*<T_\ve$, because
otherwise we have $A_K(T)\le M\ve$ for any $T\in [0, T_\ve)$,
which implies $\sup_{t\in[0,T_\ve)}|\pa u(t,x)|_K<\infty$, 
and we can extend the solution beyond $T_\ve$ by the local existence theorem.
Therefore, by the continuity of $A_K$, we have $A_K(T_\ve^*)=M\ve$;
however, since $0<T_\ve^*<T_\ve =\min\{T_\ve, \exp(\tau_0/\ve)\}$,
we can apply Lemma~\ref{Lemma5.1} to obtain $A_K(T_\ve^*)\le M\ve/2<M\ve$. 
This is a contradiction. 
Thus we conclude that \eqref{MM} is false, and we have
$T_\ve\ge \exp (\tau_0/\ve)$ for sufficiently small $\ve$. Therefore we get
$$
\liminf_{\ve\to +0} \ve\log T_\ve \ge \tau_0.
$$
Since $\tau_0 \in (0,\tau_*)$ is arbitrary, we obtain Theorem \ref{Th.1.1}.
\qed

\medskip
The rest of this section is devoted to the proof Lemma \ref{Lemma5.1}.
\medskip
 
\noindent{ \it Proof of Lemma 5.1.}
We fix $\tau_0 \in (0,\tau_*)$ and an integer $K\ge 19$.
Let $0<T < \min\{T_\ve, \exp(\tau_0/\ve)\}$
and assume that \eqref{5.1}
holds for some $M>1$ and $\ve>0$.


In the arguments below, we always suppose that $M$ 
is sufficiently large, while
$\ve$ is small enough to satisfy $M\ve \ll 1$.
$C^*$ stands for various positive constants which 
depend on $M$, $\tau_0$, and $K$,
but are independent of $T$ and $\ve$. Other positive constants, mostly denoted by $C$,
may depend on $\tau_0$ and $K$, but are independent of $M$, $T$, and $\ve$, unless otherwise stated.
\subsection{Estimates of the energy}\label{KEE1}
We fix $1/4<\gamma<1/2$.
In this subsection, we will prove that
\begin{equation}
\sup_{t\in [0, T]} \sum_{|\alpha|\le 2K} \norm{\pa^\alpha \pa u_2(t)}{L^2(\Do)}\le 
C^*\ve^{1+\gamma} \quad 
\label{ap11}
\end{equation}
holds for sufficiently small $\ve$, where $\pa=(\pa_t, \nabla_x)$, and $\alpha$ is a multi-index.
For $0\le m\le 2K$, we define
$$
z_{m}(t)=\left(\sum_{p=0}^{2K-m} \norm{\pa_t^p \pa u_2(t)}{H^m(\Do)}^2\right)^{1/2}.
$$
Then \eqref{ap11} follows from 
\begin{equation}
z_{m}(t) \le C^*\ve^{1+\gamma} \quad\text{for $t \in [0,T]$ and $0\le m\le 2K$.}
\label{bird}
\end{equation}

First we evaluate $z_0(t)$.
Let $0\le p\le 2K$.
Observe that $H(u_1,u_2)$ consists of such terms as
$(\partial_a u_i) (\partial_b u_j) $, where $a, b=0, 1, 2, 3$ and $i, j=1, 2$ with 
$(i,j) \not= (1,1)$. 
Hence (\ref{5.1}) and (\ref{3.11-o}) with $k=2K$ yield
\begin{align*}
  |\pa_t^{p} H(u_1,u_2)(t,x)| 
\le & C\left(|\pa u_1(t,x)|_{2K}+|\pa u_2(t,x)|_K\right)
        \sum_{q=0}^{2K}|\pa_t^q\pa u_2(t,x)| \\
\le & CM \ve (1+t)^{-1} 
     \sum_{q=0}^{2K} |\pa_t^{q} \pa u_2(t,x)|
\end{align*}
for $(t,x) \in [0,T] \times \overline{\Omega}$.
Therefore there exists a positive constant $C$ such that
\begin{equation}
\|\pa_t^{p} H(u_1, u_2)(t)\!:\!{L^2(\Do)}\| 
\le C M\ve (1+t)^{-1} z_{0}(t).
  \label{3.33}
\end{equation}
By \eqref{3.11++} with $\lambda=\gamma$, 
there exists a positive constant $C$ such that
\begin{align}
\|\pa_t^{p} E(u_1)(t)\!:\!{L^2(\Do)}\| 
\le C \ve^{1+\gamma} (1+t)^{-(3/2)+\gamma}.
  \label{3.33+}
\end{align}
Since (\ref{Rbc}) gives
$\partial_t^{p+1} u_2(t,x)=0$ for $(t,x)\in [0,T] \times \partial \Do$,
it follows from \eqref{3.33}, \eqref{3.33+}, and 
the energy inequality for the wave equation that 
\begin{equation}\nonumber
z_{0}(t) \le z_0(0)+\int_0^t\left(
 C_0 M\ve (1+s)^{-1}z_0(s) 
{}+C'\ve^{1+\gamma}(1+s)^{-(3/2)+\gamma}\right)ds
\end{equation}
with appropriate positive constants $C_0$ and $C'$.
The Gronwall lemma yields
\begin{align*}  
z_{0}(t) \le \left( z_{0}(0)+C'\frac{2\ve^{1+\gamma}}{1-2\gamma+2C_0M\ve}\right) 
(1+t)^{C_0 M\ve}.
\end{align*}
Since $z_0(0) \le C\ve^2$ and $\ve \log(1+T) <C(1+\tau_0)$, 
we obtain 
\begin{equation} \label{m=0}
z_0(t)\le C\ve^{1+\gamma} e^{CM(1+\tau_0)}\le C^*\ve^{1+\gamma},
\end{equation}
and \eqref{bird} for $m=0$ is proved.

Next suppose $m\ge 1$.
Then, from the definition of $z_m$, we have
\begin{align*}
z_m(t)\le & C\sum_{p=0}^{2K-m}\Bigl(\norm{\pa_t^p \pa u_2(t)}{L^2(\Do)}
{}+\sum_{1\le |\alpha|\le m}
\norm{\pa_t^p \pa_x^\alpha \pa_t u_2(t)}{L^2(\Do)}\\
&\qquad\qquad\qquad{}+\sum_{1\le |\alpha|\le m}\norm{\pa_t^p \pa_x^\alpha \nabla_x u_2(t)}{L^2(\Do)}
\Bigr)\\
\le & C\bigl(z_0(t)+z_{m-1}(t)+\sum_{p=0}^{2K-m}\sum_{2\le |\alpha|\le m+1}
\norm{\pa_t^p\pa_x^\alpha u_2(t)}{L^2(\Do)}\bigr),
\end{align*}
where we have used
$$
\sum_{1\le |\alpha|\le m}
\norm{\pa_t^p \pa_x^\alpha \pa_t u_2(t)}{L^2(\Do)}
\le C \sum_{|\alpha'|\le m-1}\norm{\pa_t^{p+1}\pa_x^{\alpha'}\nabla_x u_2(t)}{L^2(\Do)}.
$$
For $2\le |\alpha|\le m+1$, (\ref{ap10}) yields
\begin{equation}\nonumber
\|\partial_t^p \pa_x^\alpha u_2(t)\!:\!{L^2(\Do)}\|
\le C( \|\Delta_x \partial_t^{p} u_2(t)\!:\!{{H}^{m-1}(\Do) }\|
+\|\nabla_{\!x}\, \partial_t^{p} u_2(t)\!:\!{L^2(\Do)}\|).
\end{equation}
For $0\le p \le 2K-m$, we see 
that the second term on the right-hand side in the above inequality is bounded by $z_0(t)$;
using \eqref{Req}, the first term in the above is estimated by
$$
C( \|\partial_t^{p+2} u_2(t)\!:\!{{H}^{m-1}(\Do) }\|
+\|\partial_t^{p} (H(u_1, u_2)-E(u_1))(t)\!:\!{H^{m-1}(\Do)}\|),
$$
whose first and second terms are bounded by $z_{m-1}(t)$
and 
$$
 C M\ve (1+t)^{-1} z_{m-1}(t)+C\ve^{1+\gamma} (1+t)^{-(3/2)+\gamma}
$$
for $0\le p\le 2K-m$, respectively
(cf.~\eqref{3.33} and \eqref{3.33+}). In conclusion, we get
\begin{equation}\label{nanana}
 z_m(t) \le C\bigl(z_{m-1}(t)+z_0(t)+\ve^{1+\gamma}\bigr)
\end{equation}
for $m\ge 1$.
Therefore, by \eqref{m=0}
we obtain \eqref{bird} by the inductive argument with respect to $m$.

\subsection{Estimates of the generalized energy}\label{KEE2}
In this subsection we evaluate the $L^2(\Do)$-norm of the generalized derivative $\pa Z^\alpha u_2$
for $|\alpha| \le 2K-1$.
The difficulty here is that $\pa_t Z^\alpha u_2$ 
does not necessarily vanish on the boundary $\pa \Omega$.

Here we introduce notation: For a non-negative integer $m$, we define
\begin{align*}
\|\pa u_2(t)\|_m=&\left(\sum_{|\alpha|\le m}\sum_{a=0}^3 \norm{(Z^\alpha \pa_a u_2)(t)}{L^2(\Omega)}^2\right)^{1/2},\\
\TN{u_2(t)}{m}=& \left(\sum_{|\alpha|\le m}\sum_{a=0}^3\norm{(\pa_a Z^\alpha u_2)(t)}{L^2(\Omega)}^2
\right)^{1/2}.
\end{align*}
Note that we have
\begin{equation}
\label{EquiN}
C_m^{-1} \|\pa u_2(t)\|_m\le \TN{u_2(t)}{m}\le C_m \|\pa u_2(t)\|_m,
\end{equation}
where $C_m$ is a positive constant depending only on $m$.

Let $\alpha$ be a multi-index satisfying $0\le |\alpha|\le 2K-1$.
It follows from \eqref{Req} that
\begin{align}\label{ene2}
& \frac12\frac{d}{dt} 
 \int_{\Do} \left(|\partial_t Z^\alpha u_2|^2+|\nabla_{\!x}\, 
Z^\alpha u_2|^2
  \right)\,dx
\\ \nonumber
& \qquad =\int_{\Do} \left(Z^\alpha \bigl(H(u_1, u_2)-E(u_1)\bigr)\right)
\left(\partial_t Z^\alpha u_2\right)\,dx
\\ \nonumber
& \qquad \quad
 +\int_{\partial \Do} (\nu\cdot \nabla_{\!x}\, Z^\alpha u_2)\,(\partial_t
  Z^\alpha u_2)\,dS,
\end{align}
where $\nu=\nu(x)$ is the unit outer normal vector at $x \in \partial \Do$,
and $dS$ is the area element on $\partial \Do$.
Similarly to \eqref{3.33} and \eqref{3.33+},  
we obtain
\begin{align}\label{ene3}
& \|Z^{\alpha} (H(u_1, u_2)-E(u_1))(t)\!:\!{L^2(\Do)}\| 
\\ 
& \qquad \le CM\ve (1+t)^{-1} \|\pa u_2(t)\|_{m}
  +C \ve^{1+\gamma} (1+t)^{-(3/2)+\gamma}
\nonumber
\end{align}
for $|\alpha|\le m \le 2K-1$. 
From (\ref{ene2}) and (\ref{ene3}),
there exist two positive constants $C_1$ and $C_2$
such that
\begin{align}
\label{Main01}
\frac{d}{dt}\TN{u_2(t)}{m}^2\le & C_1 M\ve (1+t)^{-1}\TN{u_2(t)}{m}^2\\
& {}+C_2\ve^{1+\gamma}(1+t)^{-(3/2)+\gamma}\TN{u_2(t)}{m}+C_2J_m(t)
\nonumber
\end{align}
for $m\le 2K-1$, where 
$$
J_m(t)=\int_{\pa \Omega} |\pa u_2(t,x)|_m^2 dS.
$$

Since we have $\partial \Do \subset  B_{1}(0)$, we get
$$
|Z^\alpha \pa  u_2(t,x)| 
 \le C\sum_{|\beta| \le |\alpha|} |\partial^\beta\pa u_2(t,x)|
$$
for $(t,x)\in [0,T] \times \partial \Do$. Therefore,
by the trace theorem and \eqref{ap11} we obtain
\begin{equation}
\label{Boundary01}
J_{2K-1}(t)\le C \sum_{|\beta| \le 2K} 
\|\partial^\beta \partial u_2(t)\!:\!{L^2(\Omega\cap  B_{2}(0))}\|^2
\le C^*\ve^{2+2\gamma}.
\end{equation}
By Young's inequality, there exists a positive constant $C'$ such that
\begin{align}
\label{Young01}
& C_2\ve^{1+\gamma}(1+t)^{-(3/2)+\gamma}\TN{u_2(t)}{m} \\
& \qquad \le \frac{1}{8}(1+t)^{-1}\TN{u_2(t)}{m}^2+C'\ve^{2+2\gamma}(1+t)^{-2+2\gamma}
 \nonumber
\end{align}
for $m\le 2K-1$.
We suppose that $\ve$ is small enough to satisfy $C_1M\ve\le 1/4$. 
Then it follows from \eqref{Main01}, \eqref{Boundary01}, and \eqref{Young01} that
\begin{align*}
\frac{d}{dt}\TN{u_2(t)}{2K-1}^2 \le &
\frac{1}{2} (1+t)^{-1} \TN{u_2(t)}{2K-1}^2\\
&{}+C'\ve^{2+2\gamma} (1+t)^{-2+2\gamma}+C^*\ve^{2+2\gamma}, 
\end{align*}
from which we obtain
\begin{align*}
(1+t)^{-1/2}\TN{u_2(t)}{2K-1}^2\le  \TN{u_2(0)}{2K-1}^2+C^*\ve^{2+2\gamma}(1+t)^{1/2}.
\end{align*}
Therefore we get
\begin{equation}\label{ap15}
\TN{u_2(t)}{2K-1}\le C^* \ve^{1+\gamma} (1+t)^{1/2}
\quad\text{for}\ t \in [0,T].
\end{equation}

Now we are going to estimate $J_{2K-7}(t)$ to obtain a better estimate for $\TN{u_2(t)}{2K-7}$.
This time we employ the pointwise estimate to get a bound of $J_{2K-7}(t)$:
In the following, we set $r=|x|$. We define
$$
{\mathcal W}_-(t,r)=\min\{\jb{r}, \jb{t-r}\}.
$$
Note that
$$
 \jb{r}^{-1}\jb{t-r}^{-1}\le C \jb{t+r}^{-1} {\mathcal W}_-^{-1}(t,r).
$$
In view of \eqref{EquiN},
by (\ref{ap21}) 
and (\ref{ap15}) we have
\begin{align}\label{ap21-}
 \jb{r} |\partial u_2(t,x)|_{2K-3} 
 \le  C^*\ve^{1+\gamma} (1+t)^{1/2}.
\end{align}
It follows from \eqref{3.11-o}, \eqref{5.1}, 
and \eqref{ap21-} that 
$$
|H(u_1, u_2)(t,x)|_{2K-3} 
\le C^* \ve^{2+\gamma} (1+t)^{1/2} \jb{r}^{-1} \jb{t+r}^{-1}
{\mathcal W}_-^{-1}(t,r).
$$
Thus we find
\begin{equation}
\label{error0}
 \|H(u_1, u_2)(t)\!:\! N_{2K-3}(W_{1,1})\|
 \le C^*  \ve^{2+\gamma} (1+t)^{1/2}.
\end{equation}
We fix $0<\mu\ll 1$.
Recalling \eqref{SupportA}, from \eqref{3.11+} with $\lambda=\gamma$ 
we get
\begin{align*}
|E(u_1)(t,x)|_{2K-3}\le & C \ve^{1+\gamma}\jb{t+r}^{-2+\gamma-\mu}\jb{t-r}^{-1+\mu}\\
\le & C\ve^{1+\gamma}(1+t)^{\gamma}\jb{r}^{-1}\jb{t+r}^{-1-\mu}\jb{t-r}^{-1+\mu}
\end{align*}
for $(t,x)\in \supp E(u_1)$,
which yields
\begin{align}\label{error1}
\|E(u_1)(t)\!:\! N_{2K-3}(W_{{1+\mu}, {1-\mu}})\|
 \le C  \ve^{1+\gamma}(1+t)^{\gamma}.
\end{align}
Applying \eqref{ba4} 
with $h_1=H(u_1, u_2)$, $h_2= -E(u_1)$, $\rho=1$, $\mu_1=0$ and 
$\mu_2=\mu(>0)$, we obtain
\begin{align}
\label{KKK01}
 & \jb{x}\jb{t-|x|} |\pa u_2(t,x)|_{2K-7}
\\
 & \qquad 
\le C\ve^2+C^* \ve^{2+\gamma} (1+t)^{1/2} \log(2+t)
   +C\ve^{1+\gamma}(1+t)^{\gamma}.
\nonumber
\end{align}
Since $\ve\log (2+t)\le C(1+\tau_0)$ and $\gamma<1/2$, we obtain
\begin{align}
|\pa u_2(t,x)|_{2K-7}\le & C^* \ve^{1+\gamma}(1+t)^{-1/2}
\label{KataKata04}
\end{align}
for $(t,x)\in [0, T] \times \overline{\Do}$.
Since $\pa \Omega$ is bounded, from \eqref{KataKata04} we get
\begin{align}
\label{LEK5}
J_{2K-7}(t)
\le & C\norm{|\pa u_2(t)|_{2K-7}}{L^\infty(\pa\Do)}^2
\le C^*  \ve^{2+2\gamma} (1+t)^{-1}.
\end{align}

Now suppose that $\ve$ is small enough to satisfy $C_1M\ve\le 1/8$.
Then it follows from \eqref{Main01}, \eqref{Young01}, and \eqref{LEK5} that
\begin{align}
\label{Main02}
\frac{d}{dt}\TN{u_2(t)}{2K-7}^2 \le &
\frac{1}{4}(1+t)^{-1} \TN{u_2(t)}{2K-7}^2 
\\ 
& {}+C' \ve^{2+2\gamma} (1+t)^{-2+2\gamma}+C^* \ve^{2+2\gamma} (1+t)^{-1}.
\nonumber
\end{align}
From \eqref{Main02} we obtain
$$
(1+t)^{-1/4}\TN{u_2(t)}{2K-7}^2\le \TN{u_2(0)}{2K-7}^2+C^*\ve^{2+2\gamma},
$$
which leads to
\begin{equation}
\label{ap17}
\TN{u_2(t)}{2K-7}\le C^*\ve^{1+\gamma}(1+t)^{1/8}. 
\end{equation}

We repeat the above procedure once again to estimate $J_{2K-13}(t)$, 
with replacing \eqref{ap15} by \eqref{ap17}:
By \eqref{ap21} and \eqref{ap17}, we get 
\begin{equation}
\jb{r}|\pa u_2(t,x)|_{2K-9} \le C^*\ve^{1+\gamma}(1+t)^{1/8},
\end{equation}
which yields
\begin{equation}
\label{error3}
 \|H(u_1, u_2)(t)\!:\! N_{2K-9}(W_{1,1})\|
 \le C^*  \ve^{2+\gamma} (1+t)^{1/8},
\end{equation}
similarly to \eqref{error0}.
Applying \eqref{ba4}, it follows from \eqref{error1} and \eqref{error3} that
\begin{align*}
 & \jb{x}\jb{t-|x|} |\pa u_2(t,x)|_{2K-13}
\\
 & \qquad 
\le C\ve^2+C^* \ve^{2+\gamma} (1+t)^{1/8} \log(2+t)
   +C\ve^{1+\gamma}(1+t)^{\gamma}.
\end{align*}
Since we have $(1+t)^{1/8}\log(2+t)\le C(1+t)^\gamma$ for $\gamma>1/4$,
we obtain
\begin{align}
|\pa u_2(t,x)|_{2K-13}\le & C^* \ve^{1+\gamma}(1+t)^{-1+\gamma},
\label{KataKata05a}
\end{align}
which leads to
\begin{equation}
\label{Boundary03}
J_{2K-13}(t)\le C^* \ve^{2+2\gamma}(1+t)^{-2+2\gamma}.
\end{equation}

By \eqref{Main01} and \eqref{Boundary03}, we obtain
\begin{align}
\label{Main02a}
\frac{d}{dt}\TN{u_2(t)}{2K-13}^2\le & 
C_1 M\ve (1+t)^{-1}\TN{u_2(t)}{2K-13}^2\\
& {}+C_2\ve^{1+\gamma}(1+t)^{-(3/2)+\gamma}\TN{u_2(t)}{2K-13}
\nonumber\\
& {}+C^*\ve^{2+2\gamma}
(1+t)^{-2+2\gamma}.
\nonumber
\end{align}
We put $I_{\ve}(t)=\sqrt{\ve^{2+2\gamma}+\TN{u_2(t)}{2K-13}^2}$. Then, for $\ve>0$, we have
$$
\frac{d}{dt}I_{\ve}(t)=\left(2I_{\ve}(t)\right)^{-1} \frac{d}{dt} \left(I_{\ve}(t)\right)^2
=\left(2I_{\ve}(t)\right)^{-1}\frac{d}{dt}\TN{u_2(t)}{2K-13}^2.
$$
Hence \eqref{Main02a} yields
\begin{align}
\label{Main03}
\frac{d}{dt}I_{\ve}(t)
\le & \frac{C_1}{2} M\ve (1+t)^{-1} I_{\ve}(t) \\
& {}+\frac{C_2}{2}\ve^{1+\gamma}(1+t)^{-(3/2)+\gamma}+\frac{C^*}{2}\ve^{1+\gamma}(1+t)^{-2+2\gamma},
\nonumber
\end{align}
where we have used $\TN{u_2(t)}{2K-13}\le I_\ve(t)$ to obtain the first and second terms on the right-hand side 
of \eqref{Main03},
while $\ve^{1+\gamma}\le I_\ve(t)$ has been
 used to obtain the third term.
Since $-2+2\gamma<-(3/2)+\gamma$, by \eqref{Main03} we get
\begin{equation}
\label{C03}
(1+t)^{-C_1'M\ve}I_{\ve}(t)\le  I_{\ve}(0)+C^*\ve^{1+\gamma}\int_0^t (1+s)^{-C_1'M\ve-(3/2)+\gamma} ds,
\end{equation}
where $C_1'=C_1/2$.
It is easy to see that 
\begin{equation}
\label{C01}
I_{\ve}(0)\le \sqrt{\ve^{2+2\gamma}+\TN{u_2(0)}{2K-13}^2}\le \sqrt{\ve^{2+2\gamma}+C\ve^4}
\le C\ve^{1+\gamma}.
\end{equation}
A direct calculation yields
\begin{equation}
\label{C02}
\int_0^t (1+s)^{-C_1'M\ve-(3/2)+\gamma} ds\le \frac{2}{1-2\gamma}.
\end{equation}
Since we have $(1+t)^{C_1'M\ve}\le Ce^{CM(1+\tau_0)}$, from \eqref{C03}, \eqref{C01}, and \eqref{C02} 
we obtain
\begin{equation}
\label{C09}
\TN{u_2(t)}{2K-13}\le I_{\ve}(t)\le C^* \ve^{1+\gamma}.
\end{equation}

\subsection{Pointwise estimates}
 We see from \eqref{ap21} and \eqref{C09}
 that
\begin{align}\label{KataKata05}
 \jb{r} |\partial u_2(t,x)|_{2K-15}
\le  C^* \ve^{1+\gamma}.
\end{align}
It follows from \eqref{3.11-o}, \eqref{5.1}, 
and \eqref{KataKata05} that 
\begin{align}
|Z^\alpha H(u_1, u_2)(t,x)| 
\le C^*  \ve^{2+\gamma} \jb{r}^{-1} \jb{t+r}^{-1} 
   {\mathcal W}_-^{-1}(t,r)
\nonumber
\end{align}
for $|\alpha| \le 2K-15$, which leads to
\begin{equation}
\label{error5}
 \|H(u_1, u_2)(t)\!:\! N_{2K-15}(W_{1,1})\|
 \le C^*  \ve^{2+\gamma}.
\end{equation}
From \eqref{3.11+} with $\lambda=0$, we obtain 
$$
|E(u_1)(t,x)|_{2K-15}\le C\ve\jb{t+r}^{-2-\mu}\jb{t-r}^{-1+\mu},
$$
which yields
\begin{equation}
\label{error6}
 \|E(u_1)(t)\!:\! N_{2K-15}(W_{1+\mu,1-\mu})\|
 \le C_3 \ve
\end{equation}
for $0<\mu\ll 1$, where $C_3$ is a positive constant independent of $M$;
this independence is quite important in the following arguments.

Applying \eqref{ba4} with $h_1=H(u_1,u_2)$, $h_2= -E(u_1)$, $\mu_1=0$ and $\mu_2=\mu>0$,
from \eqref{error5} and \eqref{error6}
we obtain
\begin{equation}
\jb{x}\jb{t-|x|} |\pa u_2(t,x)|_{2K-19}\le C_4(\ve^2+C_3\ve)+C^*\ve^{2+\gamma}\log(2+t)
\label{KataKata07}
\end{equation}
for $(t,x)\in [0,T] \times \overline{\Do}$,
where $C_4$ is a positive constant independent of $M$.

Since $K\le 2K-19$, and $\ve\log(2+t)\le C(1+\tau_0)$, \eqref{KataKata07}
yields 
\begin{equation}
\label{Finale01}
\jb{x}\jb{t-|x|} |\pa u_2(t,x)|_{K}\le C_4(C_3+\ve)\ve+C^*\ve^{1+\gamma}.
\end{equation}
Finally, suppose that $M\ge 4C_4(C_3+1)$, and that a positive constant $\ve_0$ 
is small enough to satisfy $C^*\ve_0^\gamma \le M/4$. 
Then \eqref{Finale01} leads to
\begin{equation}
\label{Finale02}
\jb{x}\jb{t-|x|} |\pa u_2(t,x)|_{K}\le \frac{M\ve}{2}
\end{equation}
for $0<\ve\le \ve_0$. 
This completes the proof of Lemma \ref{Lemma5.1}.
\hfill$\qed$

%
\begin{remark}
\label{GeneralData}
\normalfont For simplicity, we have posed the initial condition \eqref{id}
with $\vec{f}=(f_0,f_1)\in \bigl(C^\infty_0(\Omega)\bigr)^2$. 
We can 
generalize the initial condition in the following way:
Suppose that $\vec{f}$ is a smooth function vanishing for large $|x|$,
and that $(\vec{f},0)$ satisfies the compatibility condition of infinite
order (see Remark~\ref{Two}). Let $\bigl(\vec{g}_\ve\bigr)_{\ve \in [0,1]}$
be a family of $C^\infty$ functions on $\overline{\Omega}$ satisfying the following properties:
\begin{itemize}
\item There exists $R>0$ such that $\vec{g}_\ve(x)=0$ for $|x|\ge R$ and $\ve\in [0,1]$.
\item For any $k\in \N$, there exists a positive constant $C_k$
such that
 \begin{equation}
  \label{NewData}
  \sum_{|\alpha|\le k} \sup_{x\in \overline{\Omega}} |\pa_x^\alpha (\vec{g}_\ve-\vec{f}\,)(x)|\le C_k\ve
 \end{equation}
for all $\ve\in [0,1]$. Thus especially we have $\vec{g}_0=\vec{f}$.
\item For each $\ve\in (0,1]$, $\bigl(\ve\vec{g}_\ve, F(\pa u)\bigr)$ satisfies the nonlinear version of compatibility condition of infinite order, that is to say, $(\pa_t^j u)(0, x)$ vanishes for $x\in \pa\Omega$
for any non-negative integer $j$, where $(\pa_t^j u)(0, x)$ is formally determined by
the nonlinear equation \eqref{eq} and the initial condition
\begin{equation}
 \label{id'} \bigl(u(0,x), \pa_t u(0,x)\bigr)=\ve \vec{g}_\ve(x),\quad x\in \Omega.
\end{equation}
\end{itemize} 
Then for the lifespan $T_\ve$ 
of the solution $u$ to the problem \eqref{eq}, \eqref{dc}, and \eqref{id'}, 
we can show that \eqref{MainInf} holds true.
Indeed, keeping the definitions of $u_1$ and $u_2$ unchanged,
we only have to replace \eqref{Rid} with 
$$
 \bigl(u_2(0,x), \pa_t u_2(0,x)\bigr)=\ve (\vec{g}_\ve-\vec{f}\,) (x).
$$
Then, since we still have \eqref{Initialu2} by \eqref{NewData}, it is easy to
modify the proof of Theorem~\ref{Th.1.1}.
\end{remark}
%
\section{Proof of Theorem~\ref{Upper}}\label{BlowUp}
Throughout this section we assume
$$
\Omega=\{x\in \R^3\,;\, |x|>1\}.
$$
We also suppose that $\vec{f}=(f_0,f_1)\in \bigl(C^\infty_0(\Omega)\bigr)^2$, and that it is radially symmetric,
i.e., $f_i(x)=f_i^*(|x|)$ for $x\in \Omega$ with some functions $f_i^*=f_i^*(r)$ for
$i=0$, $1$. Note that we may regard 
$f_i^*$ as a $C^\infty$-function on $[1,\infty)$
by putting $f_i^*(1)=0$.

We start this section with the explicit representation of
the radially symmetric solution $u$ to the Dirichlet problem
\begin{equation}
\label{RadProb}
 \begin{cases}
 \dal u(t,x)=h(t,x) & \text{for $(t,x)\in [0,\infty)\times \Omega$,}\\
 u(t,x)=0 & \text{for $(t,x)\in [0,\infty)\times \pa \Omega$,}\\
 u(0,x)=f_0(x),\ (\pa_t u)(0,x)=f_1(x) & \text{for $x\in \Omega$}.
 \end{cases}
\end{equation}
Suppose that $h$ is smooth. 
Assume that $(0,h)$ satisfies the compatibility condition of infinite order
(see Remark~\ref{Two}), and $h(t,x)=h^*(t,|x|)$ with some function $h^*=h^*(t,r)$. 
Then the solution $u$ to \eqref{RadProb} also is radially symmetric in $x$-variable,
and we can write $u(t,x)=u^*(t,|x|)$ with some function $u^*=u^*(t,r)$. 
For $i=0,1$, and $r\in \R$, we define
$$
\check{f}_i^*(r)=
\begin{cases}
rf_i^*(r) & \text{if $r>1$},\\
-(2-r)f_i^*(2-r) & \text{if $r\le 1$}.
\end{cases}
$$
Similarly we define
$$
\check{h}^*(t,r)=
\begin{cases}
rh^*(t, r) & \text{if $r>1$},\\
-(2-r)h^*(t, 2-r) & \text{if $r\le 1$}.
\end{cases}
$$
for $t\in [0,T)$. Then we have the following expression for the solution $u$:
\begin{align}
\label{RadExp}
ru^*(t,r)=&\frac{\check{f}_0^*(r-t)+\check{f}_0^*(r+t)}{2}+
\frac{1}{2}\int_{r-t}^{r+t}\check{f}_1^*(\rho)d\rho\\
&{}+\frac{1}{2}\int_{0}^t\left(\int_{r-(t-\sigma)}^{r+(t-\sigma)} \check{h}^*(\sigma,\rho)d\rho\right) d\sigma\nonumber
\end{align}
for $t\ge 0$ and $r\ge 1$.

Let $u_0$ be the solution to \eqref{RadProb} with $h\equiv 0$. Then \eqref{RadExp} implies that
${\mathcal F}_+$ defined by \eqref{DefRadiation} is given by
\begin{align}
\label{RRadExp}
{\mathcal F}_+(s,\omega)=&\frac{1}{2}\left(\frac{d}{ds}\check{f}_0^*(s)-\check{f}_1^*(s)\right).
\end{align}

Now we start the proof of Theorem~\ref{Upper}.
We just look at the region where the effect of the boundary does not appear at all.
Hence the arguments go like the case of the Cauchy problem.
\medskip\\
{\it Proof of Theorem~\ref{Upper}.}
Let $F=c(\pa_t u)^2$ with $c\ne 0$, and let $u$ be the solution to \eqref{eq}--\eqref{id},
whose lifespan is $T_\ve$.
Without loss of generality, we may assume $c>0$. Indeed we only have to
replace $u$ with $-u$ in the following arguments when $c<0$.

Assume that $f_0=0$, $f_1\in C^\infty_0(\Omega)$, and $f_1(x)=f_1^*(|x|)$ for $|x|>1$.
Suppose that $f_1^*(r)\ge 0$ for $r>1$, and $f_1^*\not\equiv 0$.
Then $\tau_*$ given by \eqref{DefTau} is
\begin{equation}\label{DefTau*}
\tau_*=\left(\sup\left\{\frac{c\check{f}_1^*(s)}{4}\,;\, s\in \R \right\}\right)^{-1}
\end{equation}
because of \eqref{RRadExp}. Observing that
$\check{f}_1^*(s)=s f_1^*(s)\ge 0$ for $s>1$, and $\check{f}_1^*(s)=-(2-s)f_1^*(2-s)\le 0$
for $s\le 1$, we find 
\begin{equation}\label{DefTau**}
\tau_*=\left(\sup\left\{\frac{cs{f}_1^*(s)}{4}\,;\, s\ge 1\right\}\right)^{-1}.
\end{equation}

Since one can show that $u$ is radially symmetric in $x$,
we can write $u(t,x)=u^*(t,|x|)$. Then
\eqref{RadExp} leads to
\begin{align}
\label{RadExp01}
r \pa_t u^*(t,r)=&\ve \frac{(r+t)f_1^*(r+t)+(r-t)f_1^*(r-t)}{2}\\
&{}+\frac{c}{2} \int_0^t(r+t-\sigma)(\pa_t u^*)^2(\sigma, r+t-\sigma)d\sigma
\nonumber\\
& {}+\frac{c}{2} \int_0^t(r-t+\sigma)(\pa_t u^*)^2(\sigma, r-t+\sigma)d\sigma
\nonumber
\end{align}
for $r-t\ge 1$.
We put $U(t,s)=(t+s)(\pa_t u^*)(t, t+s)$ for $t\ge 0$ and $s\ge 1$. Then,
by putting $r-t=s$ in \eqref{RadExp01}, we get
\begin{equation}
\label{RadExp02}
U(t,s)\ge \frac{1}{2} \ve s f_1^*(s)+\frac{c}{2}\int_0^t(\sigma+s)^{-1} U^2(\sigma, s) d\sigma=:V(t,s)
\end{equation}
for $t\ge 0$ and $s\ge 1$. From this we obtain
\begin{equation}
\frac{\pa}{\pa t}V(t,s)=\frac{c}{2}(t+s)^{-1} U^2(t,s)\ge \frac{c}{2}(t+s)^{-1} V^2(t,s).
\end{equation}
Finally we reach at
$$
V(t,s)\ge \frac{2V(0,s)}{2-cV(0,s)\log\bigl((t+s)/s \bigr)},
$$
which immediately implies
\begin{equation}
\label{RadExp03}
U(t,s)\ge \frac{2 \ve s f_1^*(s)}{4- \ve csf_1^*(s) \log\bigl((t+s)/s \bigr)}.
\end{equation}
The right-hand side of \eqref{RadExp03} blows up to infinity as 
$$
t\to s\exp\left(\frac{4}{cs f_1^*(s)}\ve^{-1}\right)-s,
$$ unless $f_1^*(s)=0$.
Suppose that $f_1^*(r)=0$ for $r\ge R$. Then, recalling \eqref{DefTau**}, we obtain
\begin{align*}
T_\ve \le & \inf_{s\ge 1; f_1^*(s)\ne 0}\left\{s\exp\left(\frac{4}{cs f_1^*(s)}\ve^{-1}\right)-s\right\}\\
\le & R \inf_{s\ge 1; f_1^*(s)\ne 0} \exp\left(\frac{4}{cs f_1^*(s)}\ve^{-1}\right)
=R\exp (\tau_{*}\ve^{-1}),
\end{align*}
which yields \eqref{upperbound} immediately.
This completes the proof.
\qed
\section*{Appendix: Proof of \eqref{ba4}}
\setcounter{equation}{0}
\setcounter{lm}{0}
\renewcommand{\theequation}{A.\arabic{equation}}
\renewcommand{\thelm}{A.\arabic{lm}}
We denote the solution $u$ to the mixed problem \eqref{MP} by
$S[(\vec{f}, h)]$. We set 
$K[\vec{f}\,]=S[(\vec{f},0)]$, 
and $L[h]=S\bigl[\bigl((0,0),h\bigr)\bigr]$.
Similarly, for a function $\vec{f^0}=(f_0^0,f_1^0)$
on $\R^3$ and a function $h^0$ on $[0,\infty)\times\R^3$,
$S_0[(\vec{f^{0}}, h^{0})]$ denotes
the solution $u^{0}$ to the Cauchy problem $\dal {u^{0}}={h^{0}}$ 
in $(0,\infty)\times \R^3$ with initial data
$({u}^{0},\pa_t {u}^{0})=\vec{f^{0}}\bigl(=({f}_0^{0}, {f}_1^{0})\bigr)$
at $t=0$. 
We also define $K_0[\vec{f^{0}}]=S_0[(\vec{f^{0}},0)]$, 
and $L_0[{h}^{0}]=S_0\bigl[\bigl((0,0),{h}^{0}\bigr)\bigr]$.

Now we start the proof of \eqref{ba4}. Let $\vec{f}$ and $h(=h_1+h_2)$ be as in 
Proposition~\ref{main}.
Following \cite{KatKub08MSJ}, we use the cut-off method to prove \eqref{ba4}
(see also \cite{ShiTsu86}):
For $m>0$, let $\psi_m=\psi_m(x)$ be a radially symmetric function satisfying
$\psi_m(x)=0$ for $|x|\le m$, and $\psi_m(x)=1$ for $|x|\ge m+1$; then we get
\begin{equation}
S[\Xi](t,x)=\psi_1(x)S_0[\psi_2\Xi](t,x)+\sum_{i=1}^4S_i[\Xi](t,x)
\end{equation}
for $\Xi=(\vec{f}, h)$, where
\begin{align}
S_1[\Xi](t,x)=& \bigl(1-\psi_2(x)\bigr)L\bigl[
                                          [\psi_1, -\Delta_x]S_0[\psi_2\Xi]
                                        \bigr](t,x),\\
S_2[\Xi](t,x)=& -L_0\left[ [\psi_2,-\Delta_x] L\bigl[
                                          [\psi_1, -\Delta_x]S_0[\psi_2\Xi]
                                        \bigr]
                     \right](t,x),\\
S_3[\Xi](t,x)=& \bigl(1-\psi_3(x)\bigr)S[(1-\psi_2)\Xi](t,x),\\
S_4[\Xi](t,x)=& -L_0\bigl[ [\psi_3,-\Delta_x] S[(1-\psi_2)\Xi]\bigr](t,x).                 
\end{align}
Here we have regarded $\Xi(t,x)$ as $0$ if $x\not\in \Omega$.

Observe that we have $\supp (1-\psi_m)\subset \{|x|\le m+1\}$ and 
$\supp \pa_a \psi_m\subset \{m\le |x|\le m+1\}$.

We put $\zeta_0=S_0[\psi_2\Xi]$. Then taking the support of $\pa_a\psi_1$ into account,
we get
\begin{align}
\label{A01}
\jb{x}\jb{t-|x|}^\rho \left|\pa_a\bigl(\psi_1\zeta_0\bigr)\right|_k
\le & C\jb{x}\jb{t-|x|}^\rho|\pa_a\zeta_0|_k\\
& {}+C|\pa_a\psi_1(x)|_k\jb{t}^\rho\sum_{|\beta|\le k}
|\pa^\beta\zeta_0|. \nonumber
\end{align}
To estimate the right-hand side of \eqref{A01}, we use the estimates for the Cauchy problem of the wave equation in $[0,\infty)\times \R^3$.
First we introduce some notation:
For $\nu\in \R$, we set
\begin{equation}
\label{defPhi}
{\Phi}_\nu(t,x)=
   \begin{cases}
        \langle t+|x|\rangle^{\nu} & \text{ if } \nu<0,
\\
       \left\{\log \bigg(2+\displaystyle\frac{\langle t+|x|\rangle}{\langle t-|x|\rangle}\bigg)\right\}^{-1}
         & \text{ if } \nu=0,
\\
       \langle t-|x|\rangle^{\nu}   & \text{ if } \nu>0.  
   \end{cases}
\end{equation}
For $\rho>0$ and a non-negative integer $k$, we define
\begin{equation}
{\mathcal B}_{\rho,k}[\vec{f^{0}}]=\sup_{x\in \R^3}\jb{x}^{\rho}\bigl(|f^{0}_0(x)|_k
{}+|\nabla_x f_0^{0}(x)|_k
{}+|f_1^{0}(x)|_{k}\bigr)
\end{equation}
(cf.~\eqref{HomWei}). We also define
\begin{equation}
\norm{{h}^{0}(t)}{M_k({\mathcal W})}
=\sup_{(s,t)\in[0,t]\times \R^3} \jb{x}{\mathcal W}(s,x)
|h^{0}(s,x)|_k
\end{equation}
for any non-negative function ${\mathcal W}$ (cf.~\eqref{NfW}). 
Recall the definitions \eqref{DefPsi} and \eqref{defW} 
for $\Psi_\nu$ and $W_{\nu,\kappa}$.
\begin{lm}\label{CauchyDecay}
Let $\rho>0$ and $\mu\ge 0$. Let $k$ be a non-negative integer. 
Then we have
\begin{align}
& 
\label{Asakura}
 \jb{t+|x|}\Phi_{\rho-1}(t,x)|K_0[\vec{f^{0}}](t,x)|_k
 \le C {\mathcal B}_{\rho+1,k}[\vec{f^{0}}], \\
& \label{Kat}
 \jb{t+|x|}\Phi_{\rho-1}(t,x)|L_0[h^{0}](t,x)|_k
\\
& \qquad\qquad\qquad\qquad\qquad\quad
\le C\Psi_\mu(t)
\norm{{h}^{0}(t)}{M_k(W_{\rho+\mu,1-\mu})}.
\nonumber
\end{align}
Moreover, if $\rho\ge 1$, then we have
\begin{equation} 
\label{Yok}
 \jb{x}\jb{t-|x|}^{\rho} |(\pa L_0[h^{0}])(t,x)|_k
\le C\Psi_\mu(t)\norm{{h}^{0}(t)}{M_{k+1}(W_{\rho+\mu,1-\mu})}.
\end{equation}
\end{lm}
\eqref{Asakura}, \eqref{Kat} for $\mu>0$, and \eqref{Yok} are essentially due to
\cite[Proposition~1.1]{asa}, \cite[Theorem~3.4]{Kub-Yok01}, and \cite[Proposition~3.1~(iii)]{Yok00},
respectively (see also \cite[Lemmas~3.2, 3.3, 3.4]{KatKub08MSJ}).
\eqref{Kat} for $\rho=1$ and $\mu=0$ (which is the only case that we 
actually need in this paper) is proved
in \cite[Lemma~2.6]{Kat04:02}. \eqref{Kat} for the case where $\rho>1$ and $\mu=0$ can be proved similarly.

Note that for $R>0$ and $\rho>0$ there exists a positive constant $C$ such that
\begin{equation}
\label{Equi01}
C^{-1} \jb{t}^\rho\le \jb{t+|x|}\Phi_{\rho-1}(t,x)\le C \jb{t}^\rho \quad \text{ for $t\ge 0$ and $|x|\le R$},
\end{equation}
and
\begin{equation}
\label{Equi02}
C^{-1} \jb{t}^\rho\le \jb{x}\jb{t-|x|}^\rho 
\le C \jb{t}^\rho \quad \text{ for $t\ge 0$ and $|x|\le R$}.
\end{equation}

We assume $\rho\ge 1$.
Writing $\zeta_0=K_0[\psi_2\vec{f}\, ]+L_0[\psi_2h_1]+L_0[\psi_2h_2]$,
from Lemma~\ref{CauchyDecay}, \eqref{A01}, and \eqref{Equi01} we get
\begin{equation} \label{mai01}
\jb{x}\jb{t-|x|}^\rho \left|\pa_a\bigl(\psi_1\zeta_0\bigr)(t,x)\right|_k\le C {\mathcal C}_{\rho,\mu_1,\mu_2,k+1}(t),
\end{equation}
where 
$$
{\mathcal C}_{\rho,\mu_1,\mu_2,k}(t)={\mathcal A}_{\rho+2,k}[\vec{f}\,]+C\sum_{j=1}^2 
  \Psi_{\mu_j}(t) \|h_j(t)\!:\!{N_{k}(W_{\rho+\mu_j,1-\mu_j})}\|.
$$
Similarly we have
\begin{equation} \label{mai01b}
\jb{t+|x|}\Phi_{\rho-1}(t,x)\left|\zeta_0(t,x)\right|_k\le C {\mathcal C}_{\rho,\mu_1,\mu_2,k}(t).
\end{equation}

To estimate $S_i$ for $1\le i\le 4$, we use 
\cite[Lemma~4.1]{KatKub08MSJ},
whose first part essentially follows from the local energy decay:
Since the obstacle ${\mathcal O}$ is assumed to be non-trapping, 
it is ``admissible'' in the sense of Definition 1.2 in \cite{KatKub08MSJ}
for $\ell=0$ and any $\gamma_0>0$ (for the proof, see \cite[Appendix B]{KatKub08MSJ}).
Therefore, (i) and (ii) of Lemma~4.1 in \cite{KatKub08MSJ} with $c=1$
lead to the following:
\begin{lm}\label{LocalEnergy} Let $M>1$, and $\rho>0$. We define
$\Omega_R=\Omega\cap B_R(0)$ for $R>1$.
\\
{\rm (1)} 
Let $\chi$ be a smooth radially symmetric function satisfying $\supp \chi\subset B_M(0)$.
For $\Xi=(\vec{f}, h)$ satisfying the compatibility condition and vanishing
for $|x|\ge R(>1)$, we have
\begin{align}
\label{61-1}
\jb{t}^\rho|\chi S[\Xi](t,x)|_k\le & C{\mathcal A}_{\rho+1, k+1}[\vec{f}\,]\\
& {}+C\sum_{|\beta|\le k+1}\sup_{(s,x)\in[0,t]\times \Omega_R}\jb{s}^\rho|\pa^\beta h(s,x)|.
\nonumber
\end{align}
 \smallskip
{\rm(2)} If $\supp \vec{f^{0}}\cup \supp h^{0}(t,\cdot)\subset 
              \overline{B_R(0)\setminus B_1(0)}$ for any $t\in [0,T)$ with some $R>1$, then
we have
\begin{align}
\label{61-3}
& \jb{x}\jb{t-|x|}^\rho|(\pa S_0[(\vec{f^{0}}, h^{0})])(t,x)|_k \\
& \qquad \le C{\mathcal A}_{\rho+2, k+1}[\vec{f^{0}}]+C\sum_{|\beta|\le k+1} \sup_{(s,x)\in[0,t]\times \Omega_R}\jb{s}^\rho|\pa^\beta h^{0} (s,x)|.
\nonumber
\end{align}

\end{lm}
We are going to estimate $S_1$. 
If we put $\zeta_1=L\bigl[[\psi_1, -\Delta_x] \zeta_0 \bigr]$, then we have $S_1=(1-\psi_2)\zeta_1$.
If we recall \eqref{Equi01} and \eqref{Equi02}, then it follows from \eqref{61-1} 
(with $\chi=1-\psi_2$, $\vec{f}\equiv 0$, and
$h=[\psi_1, -\Delta_x]\zeta_0$)
and \eqref{mai01b} that
\begin{equation}
\label{mai02}
\jb{x}\jb{t-|x|}^\rho |S_1(t,x)|_{k+1}\le C {\mathcal C}_{\rho,\mu_1,\mu_2,k+3}(t).
\end{equation}

We set $\zeta_2=[\psi_2, -\Delta_x]\zeta_1$. 
Similarly to \eqref{mai02}, from \eqref{61-1} we get
\begin{equation}
\label{mai02b}
\jb{t}^\rho |\zeta_2(t,x)|_{k+1} \le C {\mathcal C}_{\rho,\mu_1,\mu_2,k+4}(t).
\end{equation}
Observing that $S_2=-L_0[\zeta_2]$, from \eqref{mai02b} and \eqref{61-3} we get
\begin{equation}
\label{mai03}
\jb{x}\jb{t-|x|}^\rho |\pa S_2(t,x)|_{k}\le C {\mathcal C}_{\rho,\mu_1,\mu_2,k+4}(t).
\end{equation}

Setting $\zeta_3=S[(1-\psi_2)\Xi]$, we have $S_3=(1-\psi_3) \zeta_3$.
Since $\mu_i\ge 0$ for $i=1, 2$, we get
$$
\jb{s}^\rho\le \jb{x}\jb{s+|x|}^{\rho+\mu_i}{\mathcal W}_-(s,|x|)^{1-\mu_i}=\jb{x}W_{\rho+\mu_i,1-\mu_i}(s,x),
$$
where ${\mathcal W}_-(t,r)=\min\{\jb{r}, \jb{t-r} \}$.
Therefore
\eqref{61-1} leads to
\begin{equation}
\label{mai04}
\jb{x}\jb{t-|x|}^\rho|S_3(t,x)|_{k+1}\le C {\mathcal C}_{\rho,\mu_1,\mu_2,k+2}(t)
\end{equation}
by virtue of \eqref{Equi02}.

Similarly to \eqref{mai04}, \eqref{61-1} also yields
\begin{equation}
\label{mai04b}
\jb{t}^\rho|[\psi_3, -\Delta_x]\zeta_3(t,x)|_{k+1}\le C {\mathcal C}_{\rho,\mu_1,\mu_2,k+3}(t).
\end{equation}
Since $S_4=-L_0\bigl[[\psi_3,-\Delta_x]\zeta_3\bigr]$,
from \eqref{mai04b} and \eqref{61-3} we obtain
\begin{equation}
\label{mai05}
\jb{x}\jb{t-|x|}^\rho|\pa S_4(t,x)|_{k}\le  C {\mathcal C}_{\rho,\mu_1,\mu_2,k+3}(t).
\end{equation}

Finally we obtain \eqref{ba4} from 
\eqref{mai01}, \eqref{mai02}, \eqref{mai03}, \eqref{mai04}, and \eqref{mai05}.
This completes the proof.
\qed
\begin{center}
{\bf Acknowledgments}
\end{center}
The first author is partially supported by 
Grant-in-Aid for Scientific Research (C) (No.~20540211), JSPS.
The second author is partially supported by 
Grant-in-Aid for Challenging Exploratory Research (No.~22654017), JSPS 
and by Grant-in-Aid for Scientific Research (S) (No.~20224013), JSPS.


\begin{thebibliography}{99}
\bibitem{Ali01} S.~Alinhac,
{\it The null condition for quasi linear wave equations in two space dimensions I},
Invent.~Math. {\bf 145} (2001), 597--618.
\bibitem{Ali01:02} S.~Alinhac, 
  {\it The null condition for quasi linear wave equations
               in two space dimensions II},
  {Amer.~J.~Math.} {\bf 123} (2001), 1071--1101.
\bibitem{asa}F.~Asakura,
{\it Existence of a global solution to a semi-linear wave equation
with slowly decreasing initial data in three space dimensions},
 Comm.~Partial Differential Equations
{\bf 11} (1986), 1459--1487.


\bibitem{God93}P.~Godin,
{\it Lifespan of semilinear wave equations in two space dimensions}, 
Commun.~in Partial Differential Equations {\bf 18}
(1993), 895--916.

\bibitem{God08}P. Godin,
{\it
The lifespan of solutions of exterior radial quasilinear Cauchy-Neumann problems},
 J. Hyperbolic Differ. Equ.
{\bf 5} (2008), 519--546.
\bibitem{Hoe87} L.~H\"ormander,
{\it
The lifespan of classical solutions of nonlinear hyperbolic equations},
Lecture Note in Math.,
{\bf 1256}, Springer, Berlin, (1987), 241--280. 

\bibitem{Hoe97} L.~H\"ormander,
{\it
\lq\lq Lectures on nonlinear hyperbolic differential equations"},
Math\'ematiques \& Applications, {\bf 26},
Springer-Verlag, Berlin, 1997. 
\bibitem{Hos00} A.~Hoshiga, 
{\it The lifespan of solutions to quasilinear hyperbolic systems in two-space dimensions}, Nonlinear Anal. {\bf 42} (2000), Ser.~A: Theory Methods, 543--560.
\bibitem{Joh87} F.~John, 
{\it Existence for large times of strict solutions of nonlinear 
wave equations in three space dimensions for small initial data}, 
Comm.~Pure Appl.~Math. {\bf 40} (1987), 79--109.

\bibitem{Kat04:02}S.~Katayama, 
    {\it Global and almost--global existence for
                systems of nonlinear wave equations with different
                propagation speeds}, 
  Diff.~Integral Eqs. {\bf 17} (2004), 1043--1078.

\bibitem{KatKub08MSJ}S.~Katayama and H.~Kubo,
{\it An alternative proof of global existence for nonlinear 
wave equations in an exterior domain}, 
J.~Math.~Soc.~Japan. {\bf 60} (2008), 1135--1170.

\bibitem{kkInverse}S.~Katayama and H.~Kubo,
{\it The rate of convergence to the asymptotics for the wave
equation in an exterior domain}, to appear in Funkcial.~Ekvac.

\bibitem{KeSmiSo02}M. Keel, H. Smith and C. D. Sogge,
{\it Almost global existence for some semilinear wave equations},
 J. d'Anal. Math. {\bf 87} (2002), 265--279.


\bibitem{Kub06}H.~Kubo,
{\it
Uniform decay estimates for the wave equation in an exterior domain},
~{\rm in} \lq\lq Asymptotic analysis and singularities", 31--54,
Advanced Studies in Pure Mathematics 47-1, Math.~Soc. of Japan, 2007.

\bibitem{Kub-Yok01} K.~Kubota and K.~Yokoyama, 
     {\it Global existence of classical solutions to
                       systems of nonlinear wave equations with different
                       speeds of propagation},
  Japanese J.~Math. {\bf 27} (2001), 113--202.


\bibitem{LaPh}P.~D.~Lax and R.~S.~Phillips,
\lq\lq  Scattering theory", Revised Edition,
Academic Press, New York, 1989.

\bibitem{Mel79}R.~B.~Melrose, 
{\it 
Singularities and energy decay in acoustical scattering},
 Duke Math.~J. 
{\bf 46} (1979), 43--59.

\bibitem{MetNaSo05b}J.~Metcalfe, M.~Nakamura and C.~D.~Sogge,
{\it
Global existence of quasilinear, nonrelativistic wave equations
satisfying the null condition},
 Japan.~J.~Math. (N.S.) 
{\bf 31} (2005), 391--472.



\bibitem{ShiTsu86}Y.~Shibata and Y.~Tsutsumi,
{\it
On a global existence theorem of small amplitude solutions
for nonlinear wave equations in an exterior domain},
 Math.~Z
{\bf 191} (1986), 165--199.
\bibitem{Yok00}K.~Yokoyama,
   {\it Global existence of classical solutions to 
              systems of wave equations with critical nonlinearity
              in three space dimensions},
  J.~Math.~Soc.~Japan {\bf 52} (2000), 609--632.
\end{thebibliography}
\end{document}